\documentclass{article}[12pt]

\usepackage{enumerate,latexsym,fullpage}
\usepackage{amsmath,amssymb,amsthm,amsfonts,graphicx}
\usepackage{color,epic}

\def\R{{\rm\vrule depth0ex width.4pt\kern-.08em R}}
\def\reals{{\rm\vrule depth0ex width.4pt\kern-.08em R}}
\def\Prob{{\rm\vrule depth0ex width.4pt\kern-.08em P}}

\def\fpt{{\cal F}}

\newcommand{\btheta}{\mathbf{\theta}}
\newcommand{\mymap}{\psi}
\newcommand{\boldv}{v}

\newtheorem{theorem}{Theorem}
\newtheorem{lemma}[theorem]{Lemma}
\newtheorem{corollary}[theorem]{Corollary}
\newtheorem{proposition}[theorem]{Proposition}
\newtheorem{conjecture}[theorem]{Conjecture}
\newtheorem{remark}{Remark}

\newtheorem{definition}{Definition}

\newtheorem{condition}{Condition}

\begin{document}

\begin{center}
\Large{\textbf{Optimal Paths in Large Deviations of Symmetric \\  Reflected Brownian Motion in the Octant}}
\end{center}

\begin{center}
Ziyu Liang and John J. Hasenbein \\
 Graduate Program in Operations Research and Industrial Engineering \\
 Department of Mechanical Engineering \\
 University of Texas at Austin, Austin, Texas, 78712 \\
 \emph{liangzy@mail.utexas.edu, jhas@mail.utexas.edu}
 \end{center}

\begin{abstract}
We study the variational problem that arises from consideration
of large deviations for semimartingale reflected Brownian motion
(SRBM) in $\reals^3_+$. Due to the difficulty of the general problem,
we consider the case in which the SRBM has \emph{rotationally
symmetric} parameters. In this case, we are able to obtain conditions
under which the optimal solutions to the variational problem are
paths that are gradual (moving through faces of strictly increasing
dimension) or that spiral around the boundary of the octant.
Furthermore, these results allow us to provide an example for which
it can be verified that a spiral path is optimal. For rotationally symmetric
SRBM's, our results facilitate the simplification of computational methods
for determining optimal solutions to variational problems and give insight
into large deviations behavior of these processes.
\end{abstract}

\vspace{0.5cm}

\begin{center}
January 30, 2013
\end{center}

\section{Introduction and Main Results}

In this paper, we analyze the variational problem associated with the large deviations principle
for semimartingale reflected Brownian motion (SRBM) in the octant. The SRBM processes of interest arise from heavy traffic
limits of queueing network processes. Understanding the tail asymptotics of the SRBM's can aid in computing
their stationary distribution, which in turn gives insight into the behavior of the pre-limit
queueing processes.

The typical analysis of large deviations for any process can often be divided into two steps: (1) proving a
large deviations principle (LDP) and (2) analyzing the resulting variational problem. For particularly complex
variational problems, one might further subdivide step (2) into: (2a) characterizing optimal paths and
(2b) optimal path computations. Our primary interest in this paper is in step (2a), especially for SRBM's in $\reals^3_+$. To gain understanding of the difficulties of the overall investigation,
we briefly review some previous results. First, with
respect to step (1), an LDP for SRBM's in $\reals^d_+$ has only been established for special cases. For a general dimension $d$, Majewski examined the special cases of SRBM's arising from feed-forward
queueing networks \cite{maj98a} and SRBM's whose
reflection matrix is an $M$-matrix (the so-called Harrison-Reiman case)
\cite{maj98b}. Dupuis and Ramanan  \cite{dupram02} obtained an LDP for
a generalization of the Harrison-Reiman case. It should be noted that these results
still leave the LDP for $d=2$ unresolved for some parameter cases (see \cite{harhas09} for a summary). More recently, Dai and Miyazawa \cite{daimiy11} obtained exact asymptotics
for SRBM in two dimensions using moment generating functions and techniques
from complex analysis. However, the results are limited to asymptotic behavior along a ray
of the quadrant. In a related follow-up paper, Dai and Miyazawa \cite{daimiy12} provide
new insights into the results of Avram et al.\ \cite{adh99} and derive exact asymptotics for the boundary
measures of SRBM in two dimensions. 

The tasks outlined for step (2) are best explained by examining the case in two dimensions. In this setting,
Avram et al.\ \cite{adh99} and Harrison and Hasenbein \cite{harhas09} gave a complete analytical solution to the variational problem for any SRBM of interest (e.g., those
possessing a stationary distribution). The analysis was carried out in a few steps. First, three general properties
of optimal paths were established: convexity, scaling, and merging. Second, these properties were used
to conclude that only three types of optimal paths are possible. Finally, these path properties allow
the development of a complete algebraic description of the optimal paths in two dimensions.
Unfortunately, the situation in three dimensions is considerably more difficult. While the properties of convexity,
scaling, and merging still apply, they are nowhere near sufficient to characterize the possible optimal
paths. In order to attack the higher dimensional problem, we examine a special set of SRBM cases
and develop new techniques for restricting the types of paths which must be examined.

More specifically, in this paper we investigate the variational problem associated with SRBM in the positive orthant for $d=3$ in the case
in which the SRBM has either \emph{rotationally symmetric} or \emph{mirror symmetric} data (the latter is a special case of
the former). Note that neither symmetry case we analyze coincides with the much studied case of \emph{skew-symmetric} SRBM's, which
have tractable product form stationary distributions.

Our first contribution is to use the Bramson, Dai, and Harrison \cite{bdh08}
stability results to derive an appealingly simple set of stability conditions
for rotationally symmetric SRBM (see Theorem \ref{srbm_stable} in Section
\ref{sec:stable}). However, the main contribution of the paper is to clarify the
nature of optimal paths in three-dimensional variational problems and to provide
new techniques to achieve this analysis.
To best elucidate our contribution, we present
our main result now:
\begin{theorem} \label{thm:main}
Consider a rotationally symmetric variational problem, as given
in Definition \ref{def:rsvp},  arising from SRBM in $\reals_+^3$.
Suppose $\Gamma=I$ and $\theta< 0$.
Under Condition 1 in Section \ref{sec:finite}, there always exists an optimal path which is either
(a) a gradual path (a path which moves through faces of strictly increasing
dimension) or (b) a classic spiral path.
\end{theorem}
Theorem \ref{thm:supergradual} establishes part (a)
and Theorem \ref{thm:ex} establishes part (b).
 An important consequence of this result is the following:
\begin{corollary}
For the variational problem arising from SRBM in the octant, there exists
an example of an optimal spiral path.
\end{corollary}
This follows from Theorem \ref{thm:main} and the calculations in
Section \ref{sec:optspiral}.
To the best of our knowledge, this is the first time a spiral path has been shown to be
optimal for this type of variational problem. Condition 1 is somewhat complicated
and we discuss it in detail later. The most important restriction it imposes is
that the SRBM must have
reflection vectors which point ``outward.'' We believe that this condition can be relaxed. Note that the negativity condition
on the drift $\theta$ is not restrictive, since it is equivalent in our case to requiring
stability of the associated SRBM. 

An implication of these results is that they  provide the basis for tractable numerical methods for
computing optimal paths in three dimensions. Complementary to our work is a recent paper by
El Kharroubi et al.\ \cite{kytb11}, which provides some
algebraic results for paths in three dimensions. However, most of the results in \cite{kytb11} require \emph{a priori} elimination
of certain optimal path types, which we are able to provide. 
Important related computational
methodology appeared in Majewski \cite{maj98b}. If one fixes the maximum number of segments
in the search for an optimal path, Majewski's branch-and-bound algorithm can efficiently produce
the desired path.
Finally, Farlow \cite{far13} also investigates variational problems arising from
rotationally symmetric SRBM. In particular, she provides evidence that 
spiral paths cannot be optimal unless $r_1 > 1$ or $r_2> 1$. Furthermore, 
her arguments, when combined with our results, indicate that there are always spiral optimal paths in such cases. 

Our hope is that the path properties we establish can be extended beyond the symmetry cases. However, it should be noted
that the general case in three dimensions is already known to be fairly complex
and in fact our main results do not hold for all parameter cases in three dimensions.
 In \cite{dupram02}, the authors show that an optimal path to a point in the interior of the octant
may have up to five linear pieces, implying that a simple characterization of paths in the general $d=3$ case is
non-trivial. This five-piece path is depicted in Figure \ref{fig:kavita} (in the figures in this paper, dotted lines
indicate a segment contained in the interior of the octant).
\begin{figure}
\begin{centering}
\includegraphics[width=7cm]{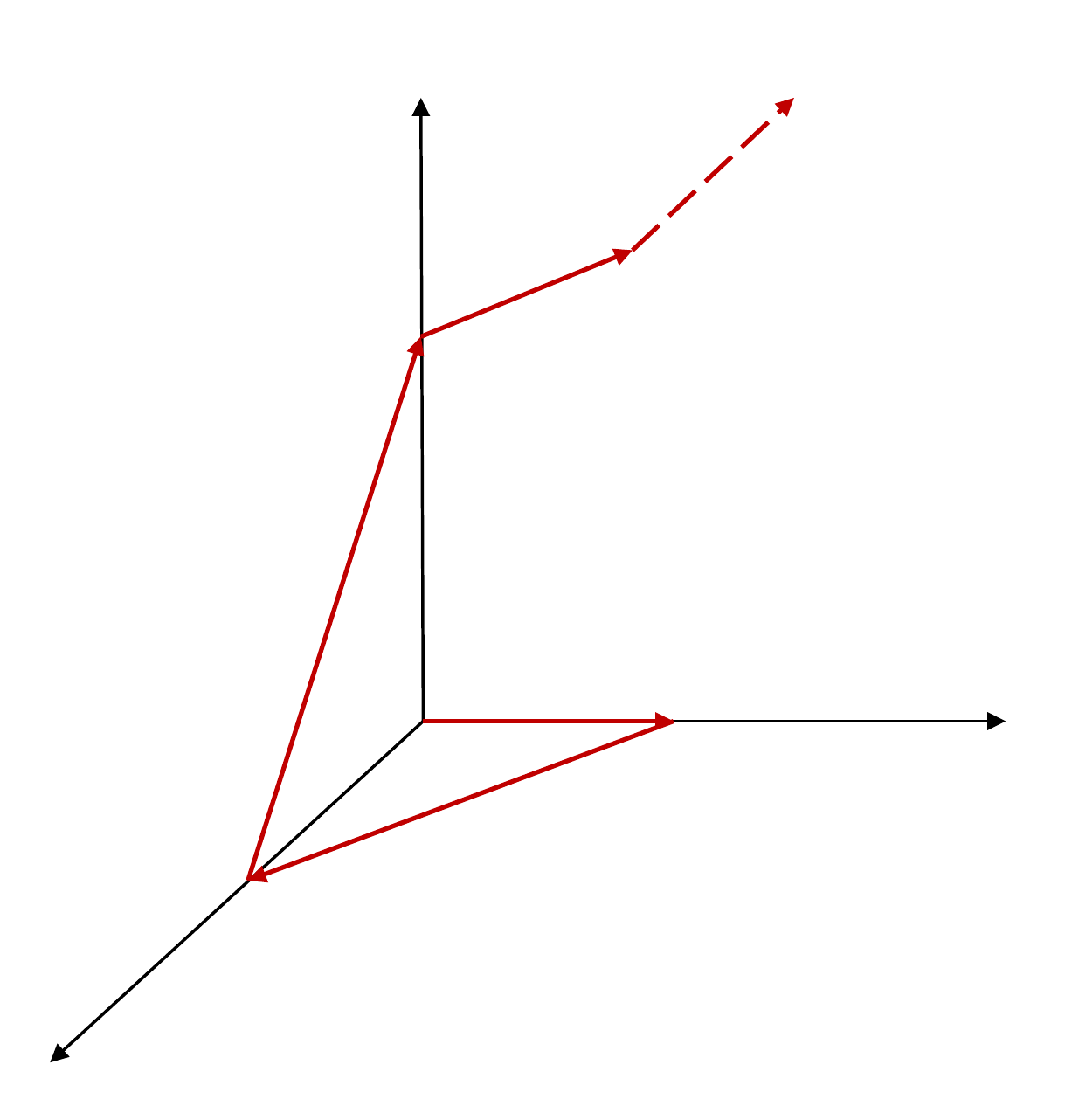}\\
\caption{An Optimal Path with Five Pieces} \label{fig:kavita}
\end{centering}
\end{figure}
Nonetheless, we believe that deriving new properties of optimal paths for special cases
is a necessary building block for solving other variational problems, and provides for a better understanding of the tail asymptotics for SRBM.

This paper is structured as follows. In Section \ref{sec:vp} we introduce the
variational problem (VP) analyzed throughout the paper. This problem arises from
studying large deviations of SRBM
in the orthant, concepts that are described in Section \ref{sec:srbm}.
Section \ref{sec:rsvp} introduces the symmetric cases of the SRBM
and VP which are of interest to us. In Section \ref{sec:stable}, we use the
framework of Bramson et al.~\cite{bdh08} to derive the stability conditions
of symmetric SRBM. The next three sections
characterize the nature of (piecewise linear) optimal paths with a finite
number of segments. Section \ref{sec:exotic} is devoted to paths with
an infinite number of pieces and it is demonstrated that only classic spiral
paths can be optimal. Finally, in Section \ref{sec:optspiral} we provide
an example of a spiral path that is indeed optimal.

\section{The Variational Problem} \label{sec:vp}

In this section, we define the variational problem of interest in this paper.
First, we give notation and definitions which follow as closely as possible
to those given in Avram et al. \cite{adh99}.

Let $d\ge 1$ be an integer and $\theta$ a constant vector in $\R^d$.
Also, $\Gamma$ is a $d\times d$
symmetric and strictly positive definite matrix,
and $R$ is a $d\times d$ matrix. The triple $(\theta, \Gamma, R)$ provides
the data to variational problems and, as described later, associated reflected
Brownian motion processes.
Throughout the paper, all vector inequalities should
be interpreted componentwise and all vectors are assumed to be column
vectors. Finally, for vectors $v\in\R^d$ and $w\in \R^d$ we define the inner
product
$$ \langle v, w \rangle = v'\Gamma^{-1} w$$
and the associated norm $||v||=\sqrt{\langle v, v \rangle}$.

In order to more easily define the
VP, we first introduce the Skorohod problem associated with the matrix $R$.
Thus, let $C([0,\infty), \reals^d)$ be
the set of continuous functions  $x: t\in [0,\infty)\to x(t)\in
\reals^d$. A function $x\in C([0,\infty), \reals^d)$ is called a path
and is often denoted by $x(\cdot)$.
We now define the Skorohod problem associated with a reflection matrix $R$.

\begin{definition}[The Skorohod Problem]
  Let $x$ be a path. An $R$-regulation of $x$ is a pair of paths
$(z, y)\in C([0,\infty), \reals^d)\times C([0,\infty), \reals^d)$
such that
\begin{eqnarray}
  \label{eqn:s1}
 &&  z(t) = x(t) + R \, y(t), \quad t\ge 0, \\
 &&  z(t) \ge 0, \quad t\ge 0, \\
 && y(\cdot) \mbox{ is non-decreasing}, \quad y(0) = 0, \label{eqn:s3}\\
 && \int_0^{\infty} z_i(s) \, dy_i(s) = 0,
\quad i = 1, \ldots, d.
\end{eqnarray}
\end{definition}

When the $R$-regulation $(y, z)$ of $x$ is unique for each $x\in
C([0,\infty), \reals^d)$, the mapping $$\psi: x\to \psi(x)=z$$
is called the reflection mapping from $C([0, \infty), \reals^d)$ to
$C([0, \infty), \reals^d_+)$.
When the triple
$(x, y, z)$ is used, it is implicitly assumed that $(y, z)$ is an
$R$-regulation of $x$.

An important issue when defining the Skorohod problem is whether a solution exists
for any given path $x$. If the reflection matrix $R$ is completely-$\cal{S}$,
as defined below, then indeed there is a solution for every $x$ with
$x(0)\ge 0$ (see Bernard and El Kharroubi~\cite{berelk90}).

\begin{definition}
\label{def:completely-S}
A $d\times d$ matrix   $R$ is said to be an \emph{$\mathcal{S}$-matrix} if
there exists  a
$u > 0$ such that $Ru>0$. The matrix $R$ is \emph{completely-$\mathcal{S}$}
if each principal submatrix of $R$ is an $\mathcal{S}$-matrix.
\end{definition}

The class of $\mathcal{P}$-matrices, defined below, also plays an important role in the
development of SRBM theory and associated variational problems.

\begin{definition}
\label{def:p-matrix}
A $d\times d$ matrix   $R$ is said to be a \emph{$\mathcal{P}$-matrix} if
all of its principal minors are positive.
\end{definition}

In addition to the issue of existence of solutions to the Skorohod problem,
there is also the matter of the uniqueness of the solution, for a given
path $x$. It is useful when defining the VP to have a notational
convention which applies when solutions are not unique. To this end,
we assume that if the Skorohod
problem is non-unique, then $\mymap(x)$ represents a set of paths (solutions)
corresponding to $x$. Furthermore, the expression $$\mymap(x) (T) = \boldv$$
indicates that there exists a $z\in \mymap(x)$ such that
$z(T) = \boldv $.

We now define the variational problem studied in this paper.
\begin{definition}[The Variational Problem]
\begin{equation} \label{eqn:varprob}
 I(\boldv) \equiv \inf_{T\ge 0} \inf_{x \in \mathcal{H}^d,
\mymap(x(\cdot))
(T) = \boldv} \; \frac{1}{2}
\int_0^T || \dot{x}(t) - \btheta||^2 \, dt
\end{equation}
where $\mathcal{H}^d$ is the space of all absolutely
continuous functions $ x(\cdot) : [0,\infty) \rightarrow
\reals^d$ which have square integrable derivatives
on bounded intervals and have $x(0)=0$.
\end{definition}

\begin{definition} \label{def:vp}
Let $v\in \R^d_+$.
If a given triple of paths $(x,y,z)$ is such that the triple satisfies
the Skorohod problem,
$z(T)=v$ for some $T\ge 0$, and
$$
\frac{1}{2} \int_0^T || \dot x(t) - \theta|| ^2\, dt = I(v),$$
then
we will call $(x,y,z)$ an \emph{optimal triple}, for VP
(\ref{eqn:varprob}), with optimal value
$I(v)$. The function
$z$ is called an \emph{optimal path} if it is the last member of an optimal triple.
Such a triple $(x, y, z)$ is also  sometimes referred to as a \emph{solution} to
the VP (\ref{eqn:varprob}). $T$ is called the \emph{optimal time} for
such a solution.
\end{definition}

\section{SRBM and Large Deviations Background}  \label{sec:srbm}

\subsection{Semi-martingale Reflected Brownian Motion}

We now define the semi-martingale reflected Brownian motion (SRBM) on the
positive orthant associated
with the data $(\theta, \Gamma, R)$. Let $\mathcal{B}$ denotes the $\sigma$-algebra of Borel subsets
of $\R^d_+$.  A triple
$(\Omega , \fpt , \{\fpt_t\} )$ is called a \emph{filtered space}
if $ \Omega $ is a set, $ \fpt $ is a $\sigma$-field of subsets of
$\Omega $, and $\{\fpt _t\} \equiv \{ \fpt _t, t \geq 0\}$ is an
increasing family of sub-$\sigma$-fields of $\fpt $, i.e., a
filtration. If, in addition, $\Prob$ is a probability measure on $
(\Omega, \fpt )$, then $ (\Omega, \fpt , \{ \fpt _t \} , \Prob )$ is
called a filtered probability space.

\begin{definition}[SRBM] \label{def:srbm}
Given a probability measure $\nu$ on $(\R^d_+, \mathcal{B})$,
 a {\em semi-martingale reflecting Brownian motion\/}
   associated with the data
 $(\theta,\Gamma,R, \nu)$
 is an $\{{\cal F}_t\}$-adapted,
  $d$-di\-men\-sion\-al process $Z$ defined on some filtered
  probability space $(\Omega, {\cal F}, \{{\cal F}_t\}, \Prob_\nu)$ such
  that
\begin{enumerate}[ (i)]
\item $\Prob_\nu$-a.s., $Z$ has continuous paths and $Z(t)\in \R^d_+$ for
  all $t\ge 0$,
\item $Z=X+RY$, $\Prob_\nu$-a.s.,
\item under $\Prob_\nu$,
  \begin{enumerate}[ (a)]
  \item $X$ is a $d$-dimensional Brownian motion with drift vector
    $\theta$, covariance matrix $\Gamma$ and $X(0)$ has distribution
    $\nu$,
  \item $\{X(t)-X(0)-\theta t, \mathcal{F}_t, t\ge 0\}$ is a
    martingale,
  \end{enumerate}
\item $Y$ is an $\{\mathcal{F}_t\}$-adapted, $d$-dimensional process
  such that $\Prob_\nu$-a.s.\ for each $j=1, \ldots, d$,
  \begin{enumerate}[ (a)]
  \item $Y_j(0)=0$,
  \item $Y_j$ is continuous and non-decreasing,
  \item $Y_j$ can increase only  when $Z$ is on the face $F_j\equiv
    \{x\in\R^d_+: x_j=0\}$, \\i.e., $\int_0^\infty Z_j(s)\, dY_j(s)=0$.
  \end{enumerate}
\end{enumerate}
An SRBM associated with the data $(\R^d_+, \theta, \Gamma
,R)$ is an $\{{ \cal F }_t\}$-adapted, $d$-dimensional process $Z$
together with a family of probability measures $\{ \Prob_x, x\in \R^d_+\}$
defined on some filtered space $(\Omega, \fpt , \{\fpt_t\})$ such that,
for each $x\in \R^d_+$, (i)-(iv) hold
with $P_{\nu} = P_x$ and $\nu$ being the point distribution at $x$.
\end{definition}

Recall that the parameters $\theta$, $\Gamma$ and $R$ are
called the \emph{drift vector}, \emph{covariance matrix}
and  \emph{reflection matrix} of the SRBM, respectively.
The results of Reiman and Williams~\cite{reiwil88} and
Taylor and Williams~\cite{taywil93} imply that the necessary
and sufficient conditions for the existence the SRBM is
that $R$ is completely-$\mathcal{S}$.

The measure $\nu$ on $(\R^d_+, \mathcal{B})$ is a stationary distribution for an
SRBM $Z$ if for each $A\in \mathcal{B}$,
\begin{equation}
  \label{eq:sd}
  \nu(A) = \int_{\R^d_+} \Prob_x\{ Z(t)\in A \} \, \nu(dx) \quad
  \mbox{for each } t\ge 0.
\end{equation}
When $\nu$ is a stationary distribution, the process $Z$ is stationary
under the probability measure $\Prob_\nu$. In our discussion below, we are
concerned only with the (unique) stationary distribution for the SRBM
with data $(\theta, R, \Gamma)$ and therefore we drop the $\nu$ notation.

\subsection{Large Deviations}  \label{sec:ldp}

The motivation for studying the variational problem
introduced in Section \ref{sec:vp} comes from the theory of
large deviations.
For SRBM's in $\reals_+^d$, we have the
following statement of the large deviations principle,
which has only been established for some special cases,
as noted in the introduction.

\begin{conjecture}[General Large Deviations Principle] \label{thm:gldp}
Consider an SRBM $Z$ with data
$(\theta, \Gamma, R)$. Suppose that $R$ is a
completely-$\mathcal{S}$ matrix and that there exists a
probability measure $\Prob_{\pi}$ under which
$Z$ is stationary. Then for every measurable $A \subset \reals^d_+$
\begin{equation} \label{eqn:conj1}
 \limsup_{u \to \infty} \frac{1}{u} \log
\Prob_\pi(Z(0) / u \in A) \; \leq \; - \inf_{v\in A^c} I(v)
\end{equation}
and
\begin{equation} \label{eqn:conj2}
\liminf_{u \to \infty} \frac{1}{u} \log
\Prob_\pi(Z(0) / u \in A) \; \geq \; - \inf_{v\in A^o} I(v)
\end{equation}
where $A^c$ and $A^o$ are respectively the closure
and interior of $A$.
\end{conjecture}
The specific connection between this LDP statement and the VP is that
the function $I(\cdot)$ appearing above is the same function which appears
in Definition \ref{def:vp}.

\section{Symmetric SRBM}   \label{sec:rsvp}

In this paper, we study solutions to the variational problem associated with the LDP
introduced in the previous section. The three-dimensional case is considerably more
difficult than the two-dimensional case and thus we confine our study to SRBM's with
some symmetry in the data. The special cases we study are called \emph{rotationally symmetric}
and \emph{mirror symmetric} and are defined below. These symmetries, while restrictive, provide
a considerable simplification of the analysis.

\begin{definition} \label{def:rsvp}
For $d=3$ the data $(\theta, \Gamma, R)$ is said to be \textbf{rotationally symmetric} if all of the following
three conditions hold:
\begin{enumerate}
\item $R$ has the form
$$ R=
                                  \left(
                                  \begin{array}{ccc}
                                   1     & r_{2} & r_{1}  \\
                                   r_{1} & 1     & r_{2}  \\
                                   r_{2} & r_{1} & 1
                                   \end{array}
                                   \right).$$
\item The drift has the form $\theta =(\theta_{0},\theta_{0},\theta_{0})'$.

\item The covariance matrix has the form
$$    \Gamma=              \left(\begin{array}{ccc}
                                                         \sigma^{2}     & \rho\sigma^{2} & \rho\sigma^{2} \\
                                                         \rho\sigma^{2} & \sigma^{2}     & \rho\sigma^{2} \\
                                                         \rho\sigma^{2} & \rho\sigma^{2} & \sigma^{2}
                                                         \end{array}
                                                         \right),
                                                         $$
where $-1 < \rho < 1$.
\end{enumerate}
\end{definition}
Some statements in the rest of the paper relate only to $R$ and in this case we call $R$ alone rotationally symmetric if and only if
$R$ has the form given in the definition above. We employ a similar convention for $\Gamma$.

\begin{definition}
For $d=3$ the data $(\theta, \Gamma, R)$ is said to be \textbf{mirror symmetric} if it is rotationally
symmetric and in addition $r_{1} = r_{2}$.
\end{definition}

Notice that the rotationally symmetric $\Gamma$ matrix also appears to be mirror symmetric.
Since covariance matrices are by definition symmetric (in the standard matrix algebra sense),
there is no sensible way to define a rotationally symmetric $\Gamma$ which is not also mirror
symmetric. For a rotationally symmetric $\Gamma$ we have the following result, proved in the
Appendix, which
will be used in demonstrating optimal path properties.
\begin{lemma} \label{lemma:matrix}
If $\Gamma$ is rotationally symmetric, then $$\Gamma^{-1}=\sigma^{-2}\left(
                                                                     \begin{array}{ccc}
                                                                       \gamma_{0} & \gamma_{1} & \gamma_{1} \\
                                                                       \gamma_{1} & \gamma_{0} & \gamma_{1} \\
                                                                       \gamma_{1} & \gamma_{1} & \gamma_{0} \\
                                                                     \end{array}
                                                                   \right),$$
with $\gamma_{0} > \gamma_{1}.$
\end{lemma}

Some readers may also be familiar with the skew-symmetry condition (see \cite{harwil87a,harwil87b})
which is
\begin{equation}  \label{eqn:skewsym}
2\Gamma = R D^{-1} \Lambda + \Lambda D^{-1} R^{\prime},
\end{equation}
where $D = \mathrm{diag}(R)$ and $\Lambda = \mathrm{diag}(\Gamma)$.
This condition is necessary and sufficient for the stationary density function of the
SRBM to admit a separable, exponential form. Our notions of symmetry do not coincide in any meaningful
way with the notion of skew-symmetry. It can be checked that rotationally symmetric SRBM data
is also skew-symmetric if and only if $r_1+r_2 = 2 \rho$.

In subsequent sections, we provide results for both SRBM and the associated variational problems.
Thus, for an SRBM with rotationally symmetric data we use the abbreviation RS-SRBM. Similarly, for
an SRBM with mirror symmetric data we use MS-SRBM. The associated variational problems take the
same data and when stating results for VPs we use the abbreviations RSVP and MSVP.

\section{SRBM Stability Conditions} \label{sec:stable}

For SRBM in three dimensions Bramson et al.~\cite{bdh08} obtained results which, in addition to
previous results, give a complete characterization
of existence and stability of SRBM. This characterization is summarized in
Figure \ref{fig:stable}. The results of this section specialize their results for RS-SRBM. First, however,
we need to define a few terms appearing in the figure.

\begin{figure}
\begin{centering}
\includegraphics[width=9cm]{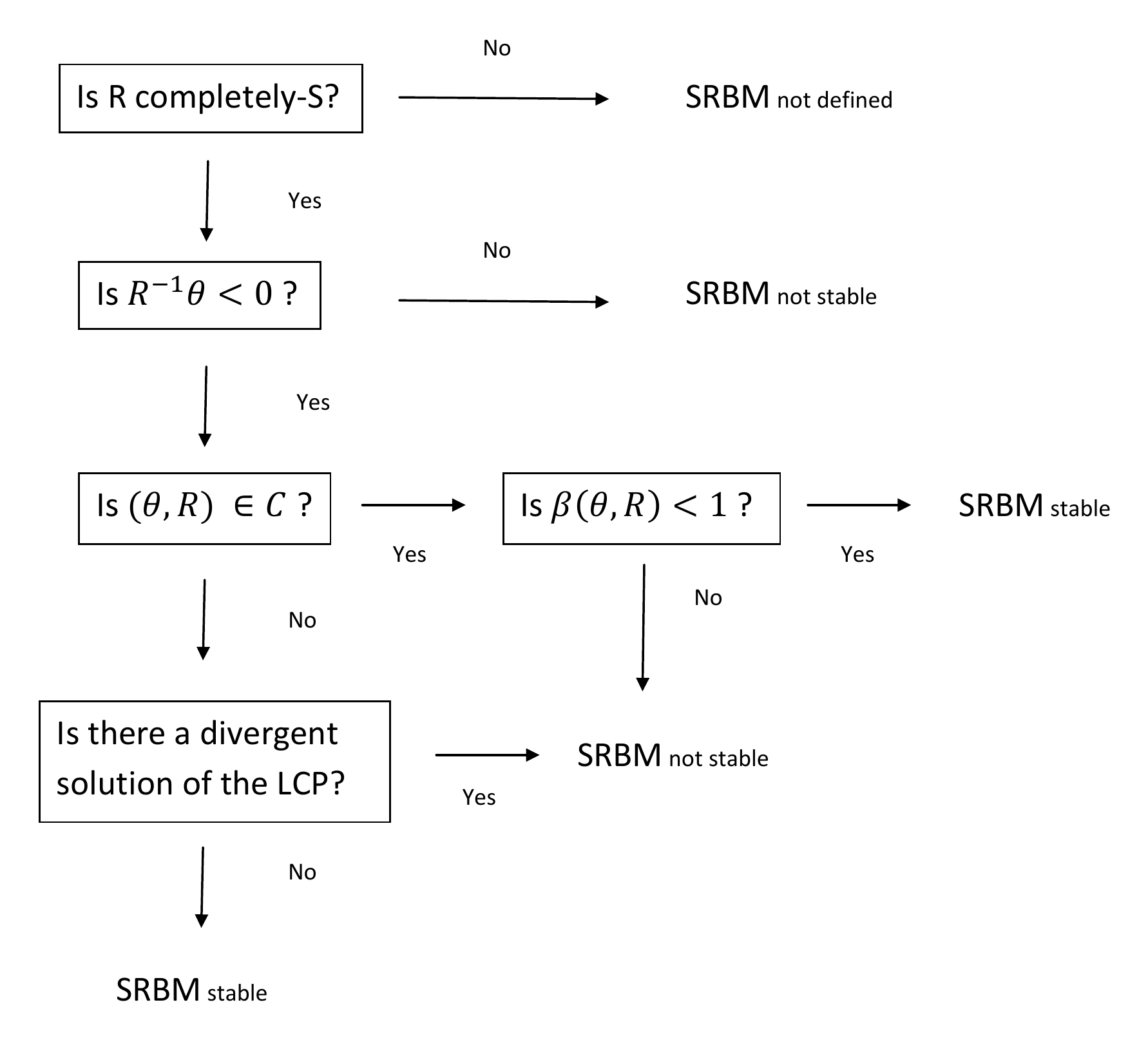}\\
\caption{Existence and Stability of SRBM in the Octant} \label{fig:stable}
\end{centering}
\end{figure}

We define the solutions to the \emph{linear complementarity problem} (LCP) in dimension $d$ as follows
(see \cite{cottle} for background). Vectors $u,v \in \reals^d$ comprise a solution to the LCP if
\begin{eqnarray*}
u, v & \ge & 0 \\
v & = & \theta + Ru \\
u \cdot v & = & 0.
\end{eqnarray*}

Using the terminology in \cite{bdh08} a solution $(u,v)$ to the LCP is called
\emph{stable} if $v=0$ and the solution is called \emph{divergent} otherwise.
The existence or non-existence of a solution to the LCP must be checked in the
bottom decision point in Figure \ref{fig:stable}.

Bramson et al.~\cite{bdh08} also define various subsets of the data pairs $(\theta, R)$
which relate to the third line of decision points in Figure \ref{fig:stable}. To avoid
overlapping notation we specialize their definitions now to the RS-SRBM case.

First, for a pair $(\theta, R)$,
\begin{eqnarray*}
C_1 & = & \{ (\theta_0, r_1, r_2) : \theta_0 < 0, r_1 < 1, r_2 >1 \} \qquad \mbox{and} \\
C_2 & = & \{ (\theta_0, r_1, r_2) : \theta_0 < 0, r_1 > 1, r_2 <1 \},
\end{eqnarray*}
with $C = C_1 \bigcup C_2$.

Next, for a data pair $(\theta, R) \in C_1$,
$$ \beta(\theta,R) = \left( \frac{1-r_2}{r_1-1} \right)^3$$
and for $(\theta, R) \in C_2$
$$ \beta(\theta,R) = \left( \frac{r_1-1}{1-r_2} \right)^3.$$
For general SRBM data $\beta(\theta,R)$ depends on $\theta$ but in the
rotationally symmetric case the dependence disappears. These definitions are
related to spiral piecewise linear solutions of the Skorohod problem.
Section 3 in \cite{bdh08} should be consulted for an in-depth explanation
of how these expressions arise.

We are now prepared to present a series of lemmas which lead to the main stability
result of this section.
The first lemma probably appears in a textbook somewhere, but we state it here and prove
it in the Appendix for completeness. For later use, note that the lemma implies that $a+b+c \not= 0$. All the results stated in this section apply to the three-dimensional case.
\begin{lemma} \label{lem:abc}
If a reflection matrix $R$ is non-singular and rotationally symmetric then its inverse must be of the form
 $$    R^{-1}=     \left(\begin{array}{ccc}
                                                                                      a & b & c \\
                                                                                      c & a & b \\
                                                                                      b & c & a
                                                                                     \end{array} \right),
                                                                                          $$
and $(a+b+c)(1+r_{1}+r_{2}) = 1.$
\end{lemma}

Since existence of an SRBM requires that $R$ is completely-$\cal{S}$, our first
task is derive a simple condition to insure that this characterization holds.

\begin{lemma} \label{comps}
Suppose a matrix $R$ is rotationally symmetric. Then $R$ being
completely-$\cal{S}$ is equivalent to $1+r_{1}+r_{2} > 0.$
\end{lemma}

\begin{proof}
First suppose $R$ is completely-$\cal{S}$ and $1+r_{1}+r_{2} \le 0$.
We derive a contradiction.
Since $R$ is completely-$\cal{S}$, there exists a vector
$u\equiv(u_{1}, u_{2}, u_{3})'>0$ such that $Ru>0.$ Summing the equations
in $Ru>0$ we have
\begin{equation} \label{ru}
(1+r_{1}+r_{2})(u_{1}+u_{2}+u_{3}) > 0.
\end{equation}
But if  $1+r_{1}+r_{2} \le 0$, then there is no $u>0$ satisfying (\ref{ru}),
which is a contradiction. So, we have proven that completely-$\cal{S}$ implies
$1+r_{1}+r_{2} > 0$ which is one direction of the equivalence.

Now assume that $1+r_{1}+r_{2} > 0$. Then note that
$u=(1,1,1)'$ satisfies $Ru>0$. This implies that $R$ is an $\cal{S}$-matrix.
We must now verify that the two-by-two principal submatrices are also $\cal{S}$-matrices.
These submatrices take the form
$$ S_1 = \left(
     \begin{array}{cc}
       1 & r_2 \\
       r_1 & 1 \\
     \end{array}
   \right)   \qquad \mbox{and} \qquad
S_2 = \left(
  \begin{array}{cc}
    1 & r_1 \\
    r_2 & 1 \\
  \end{array}
\right).$$
We prove the result for $S_1$ only, since the argument for $S_2$ is completely
analogous. Now, for $S_1$ to be an $\cal{S}$-matrix there must exist a $(u_1,u_2)'>0$
such that,
\begin{eqnarray}
u_1 + u_2r_2 & > & 0  \label{rua} \\
u_1r_1 + u_2 & > & 0. \label{rub}
\end{eqnarray}
First suppose $r_1,r_2 > 0$. Then any $(u_1,u_2)'>0$ satisfies (\ref{rua}) and (\ref{rub}). In the
cases $r_1 > 0$, $r_2 < 0$ and  $r_1 < 0$, $r_2 > 0$ it is clear that (\ref{rua}) and (\ref{rub})
for some positive $u$. In the last case, 
$r_1,r_2 < 0$, it can be checked that (\ref{rua}) and (\ref{rub}) holding
for some $u$ is equivalent to $r_1r_2<1$. This condition holds when
$r_1$ and $r_2$ are negative because $1+r_{1}+r_{2} > 0$
insures $r_1, r_2 > -1$.
Finally, if $r_1=0$ and/or $r_2=0$ then, for example $u=(1,1)$ satisfies (\ref{rua}) and (\ref{rub}).

\end{proof}

Having now dispatched with the first line in Figure \ref{fig:stable}, we present
a lemma relating to the second line.
\begin{lemma} \label{nec}
Let the data $(\theta, \Gamma, R)$ be rotationally symmetric and
let $R$ be non-singular and completely-$\cal{S}$. Then $R^{-1}\theta < 0$ is equivalent to $\theta_{0} < 0$.
\end{lemma}
\begin{proof}
Using Lemma \ref{lem:abc} the condition $R^{-1}\theta < 0$ reduces to
$$ (a+b+c) \theta_0 < 0.$$ The second part of Lemma \ref{lem:abc}
states that $(a+b+c)= [1+r_{1}+r_{2}]^{-1}$. Hence, we can rewrite the condition
as
\begin{equation} \label{rtheta}
R^{-1}\theta  =  \frac{\theta_{0}}{1+r_{1}+r_{2}} < 0.
\end{equation}
By Lemma \ref{comps}, the completely-$\cal{S}$ condition is equivalent to $1+r_{1}+r_{2}>0$.
Given this inequality, (\ref{rtheta}) is clearly equivalent to
$\theta_0<0$.
\end{proof}

Now we proceed to results involving the last two lines in Figure \ref{fig:stable}.
\begin{lemma} \label{suf}
Let the data $(\theta, \Gamma, R)$ be rotationally symmetric. Suppose further
that $R$ is non-singular, completely-$\cal{S}$, and $R^{-1}\theta < 0$.
Then the SRBM associated with
$(\theta, \Gamma, R)$ is stable iff $r_{1}+r_{2} < 2$.
\end{lemma}
\begin{proof}
Our proof relies on the results in \cite{bdh08} as depicted in Figure \ref{fig:stable}.
The assumptions of the lemma place us in the lower half of the figure.
To further partition the proof,
we divide the $(r_1,r_2)$ plane into four regions:
\begin{itemize}
  \item $C_{1} = \{(r_1,r_2) \in \reals^2 : r_{1} < 1, r_{2} > 1\}$
  \item $C_{2} = \{(r_1,r_2) \in \reals^2 : r_{1} > 1, r_{2} < 1\}$
  \item $C_{3} = \{(r_1,r_2) \in \reals^2 : r_{1} \ge 1, r_{2} \ge 1   \}  \setminus (1,1)  $
  \item $C_{4} = \{(r_1,r_2) \in \reals^2 : r_{1} \le 1, r_{2} \le 1\}  \setminus (1,1).$
\end{itemize}
We do not include the completely-$\cal{S}$ condition that $r_1 + r_2 > -1$ in this partitioning
scheme because the condition is not employed directly
in the algebraic arguments below.
Under our assumption
$R^{-1}\theta < 0$, the definitions of $C_1$ and
$C_2$ coincide with the Bramson et al.~\cite{bdh08} definitions given in
Section \ref{sec:stable}. Furthermore, note that if $(r_1,r_2) \in C_1 \bigcup C_2$
then $(\theta, R) \in C$.

The one point of the plane not included in the union of these regions is $r_1=r_2=1$. The matrix $R$ is singular in this case, which violates the assumption of the lemma.
Notice that the line $r_1+r_2=2$ bisects $C_1 \bigcup C_2$ and that $C_3$ lies
entirely above this line and $C_4$ entirely below this line.

\vspace{0.1in}

\noindent
\textbf{Case 1}: Suppose $(r_1,r_2) \in C_1 \bigcup C_2$. In this case, stability of the SRBM is equivalent
to $\beta(\theta,R) < 1$.
Now, when $(r_1,r_2) \in C_{1}$ we have, $$\beta(\theta,R)  = \left(\frac{1-r_{2}}{r_{1}-1}\right)^{3}.$$
Since the numerator and denominator are both negative for
$(r_1,r_2) \in C_{1}$, the condition $\beta(\theta,R) < 1$ is equivalent to
$1-r_{2} > r_{1}-1$, which holds iff $r_{1}+r_{2} < 2$.

Next, when $(r_1,r_2) \in C_{2}$, $$\beta(\theta,R) = \left(\frac{1-r_{1}}{r_{2}-1}\right)^{3}.$$
Again, the numerator and denominator in the last expression are both negative for
$(r_1,r_2) \in C_{2}$. Therefore, $\beta(\theta,R) < 1$ is equivalent
$1-r_{1} > r_{2}-1$, which also holds iff $r_{1}+r_{2} < 2$.

So, for all of Case 1, $r_{1}+r_{2} < 2$ is necessary and sufficient for stability.

\vspace{0.1in}

\noindent
\textbf{Case 2}: Suppose $(r_1,r_2) \in C_3 \bigcup C_4$. In this case stability of the SRBM is equivalent to
the nonexistence of a divergent solution to the LCP.

\vspace{0.1in}

\noindent
\textbf{Case 2a}: We examine the case $(r_1,r_2) \in C_3.$
Since $r_1+r_2 >2$ for all points in $C_3$ we need to show instability
for data in this region.
In the LCP, let $u = (-\theta_{0},0,0)'$ and  $v = -\theta_0 (0, r_{1}-1, r_{2}-1)'.$ This
is clearly a divergent solution to the LCP for any $(r_1,r_2) \in C_3$. Therefore the corresponding SRBM is never stable in this
case.

\vspace{0.1in}

\noindent
\textbf{Case 2b}: We examine the case $(r_1,r_2) \in C_4.$
Since $r_1+r_2 <2$ for all points in $C_3$ we need to show stability
for data in this region.

Consider a solution $u,v$ to the LCP. We show that there exist no divergent solutions for
this case.
If $u > 0$, then $v = 0$ and the solution is stable. If $u = 0$, then we must have $v = \theta < 0$ which
is not an allowable LCP solution. So if there exists a divergent solution, either one term or two terms in
$u=(u_1,u_2,u_3)'$ is positive.
Suppose one term is positive and it is $u_1$, which implies $v_1=0$. Then we have
$v_{1} = \theta_{0}+u_{1} = 0$ yielding $u_{1}= -\theta_{0}$. Therefore, in this case
the unique a solution to the LCP must be of the form $u = (-\theta_{0},0,0)',$ $v = -\theta_0 (0,r_{1}-1,r_{2}-1)',$
which violates the non-negativity condition of $v$. Exactly analogous arguments show that
neither $u_2$ nor $u_3$ can be the positive term.  So, there exist no LCP solutions in which
only one term in $u$ is positive.

Next, suppose that two terms of $u$ are positive. Again, without loss of generality,
suppose $u_{1} > 0, u_{2} > 0$, and  $u_{3} = 0$, which implies $v_{1}=v_{2}=0$.
From the LCP equations we have $u_{1}+r_{2}u_{2}+\theta_{0} = 0$ and
$u_{2}+r_{1}u_{1}+\theta_{0} = 0.$ Solving these yields:
\begin{equation} \label{u1u2}
u_{1} = \frac{-\theta_{0}(1-r_{2})}{1-r_{1}r_{2}} \qquad u_{2} = \frac{-\theta_{0}(1-r_{1})}{1-r_{1}r_{2}}.
\end{equation}
If $r_1=1$, then $r_2 <1$ and these equations
force $u_2 =0$, which contradicts our assumption on $u$. Similarly, we cannot
have $r_2=1$. So, we now assume that both $r_1$ and $r_2$ are strictly less than 1.
In this case, the solution given in (\ref{u1u2}) implies that both $u_1$ and $u_2$ are
positive. Once again using the LCP equations we obtain:
\begin{equation}
v_{3} = -\theta_0 \cdot \frac{r_{1}-r_{1}^{2}+r_{2}-r_{2}^{2}+r_{1}r_{2}-1}{1-r_{1}r_{2}}.
\end{equation}
Note that $1-r_{1}r_{2} > 0$, $-\theta_0 > 0$, and $$r_{1}-r_{1}^{2}+r_{2}-r_{2}^{2}+r_{1}r_{2}-1 \leq r_{1}+r_{2}-r_{1}r_{2}-1 = -(1-r_{1})(1-r_{2}) < 0,$$ Therefore, $v_{3}<0$ which violates the non-negativity condition in the LCP.

We have now demonstrated that no divergent LCP solutions exist when
$(r_1,r_2) \in C_4$. So, any SRBM with data
in this region is stable.
\end{proof}

Lemmas \ref{comps} through \ref{suf}
then imply simple existence and stability conditions for RS-SRBM in
three dimensions.

\begin{theorem} \label{srbm_stable}
Consider an SRBM in three dimensions with rotationally symmetric data $(\theta, \Gamma, R)$.
The necessary and sufficient conditions for existence and stability of such an SRBM are $\theta_0 <0$ and
$ -1 < r_1+r_2 <2.$
\end{theorem}

The results in \cite{kytb11}, which we shall make use of in later sections, require that $R$ be a $\cal{P}$-matrix.
The next result shows that this is not a restriction in the rotationally symmetric case, given
that we only study stable SRBM's. The proof is given in the Appendix.

\begin{theorem} \label{pands}
Suppose $R$ is rotationally symmetric and $r_1+r_2 <2$. Then
$R$ being completely-$\cal{S}$ is equivalent to $R$ being a $\cal{P}$-matrix.
\end{theorem}

\section{Optimal Path Preliminaries}

In this section we establish some notation and review the properties pertaining to optimal paths.
As much as possible we use notation which is consistent with either \cite{adh99} or \cite{kytb11}. Many of our results
rely on algebraic expressions given in \cite{kytb11}. First, we give expressions for the optimal costs
of various types of paths.

Set $I = \{1,2, \ldots, d\}$ and for $K \subset I$ define the face associated with $K$ as follows:
$$ F_K = \{ v \in \mathbb{R}_+^{d} : v_i = 0 \; \mbox{for all $i \in K$} \}.$$
When $d=3$, if $|K|=2$ then $F_K$ is an axis and if $|K|=1$ then $F_K$ is a 2-dimensional face.
\begin{definition} Let $\mathcal{H}^{d}_w$ be the modification of $\mathcal{H}^d$ such that $x(0)=w$. We define the following costs, inspired by the notation
in \cite{kytb11}.
\begin{enumerate}
\item (Direct Path Cost) For $w,v \in \mathbb{R}^{d}_{+}$, set $$\tilde{\mathcal{I}}_{0}(w,v) = \inf_{T \geq 0} \inf_{x \in \mathcal{H}^{d}_w,x(T)=v} \frac{1}{2} \int^{T}_{0} \parallel \dot{x}(t)-\theta \parallel ^{2} dt.$$

In a slight abuse of notation, we set $\tilde{\mathcal{I}}_{0}(v) : =\tilde{\mathcal{I}}_{0}(0,v).$ 

\item (One-piece Reflected Path Cost) 
Let $J$ and $K$ be subsets of $I$ with $K \subset J$ and $0 < |K| \le |J| \le d$.
For points $v \in F_K$ and $w \in F_J$, set
 $$\tilde{\mathcal{I}}_{K}(w,v) = \inf_{T \geq 0} \inf_{x \in \mathcal{H}^{d}_w, z(t) \in F_K \forall t \in [0,T], \psi(x)(T)=v} \frac{1}{2} \int^{T}_{0} \parallel \dot{x}(t)-\theta \parallel ^{2} dt.$$
 Set $\tilde{\mathcal{I}}_{K}(v)=\tilde{\mathcal{I}}_{K}(0,v).$
\item (Two-Piece Path via Face $F_K$) Let $d = 3$ and $K \subset I$ with $|K| \leq 2$. For $v \in \mathbb{R}^{3}_{+} \setminus F_K$, set $$\tilde{\mathcal{I}}^{2}_{K}(v) = \inf_{w \in F_{K}} {(\tilde{\mathcal{I}}_{K}(w) + \tilde{\mathcal{I}}_{0}(w,v))}.$$
\item (Two-Piece Path via an Axis). Let $d = 3$ and $K \subset I$ with $|K| = 2$. For $i \in K$ and $v \in F_{i}$, set $$\tilde{\mathcal{I}}^{2}_{K,i}(v) = \inf_{w \in F_{K}} {(\tilde{\mathcal{I}}_{K}(w) + \tilde{\mathcal{I}}_{i}(w,v))}.$$
\item (Three-Piece Gradual Escape Path) Let $d = 3$ and $K \subset I$ with $|K| = 2$. For $i \in K$ and $v \in int(\mathbb{R}^{3}_{+})$, set $$\tilde{\mathcal{I}}^{3}_{K,i}(v) = \inf_{u \in F_{i}} {(\tilde{\mathcal{I}}^{2}_{K,i}(u) + \tilde{\mathcal{I}}_{0}(u,v))}.$$
\end{enumerate}
\end{definition}
Each cost above corresponds to the cost for an optimal
path of a certain type, as denoted in each item in the list. These costs, and the associated paths,
are the building blocks for constructing paths which are optimal in the original
variational problem.

In \cite{adh99}, the authors established various properties of optimal paths that hold
in all dimensions. The first three items in Lemma \ref{lem:opp} restate those properties. We add a fourth
property for RSVPs and MSVPs in three dimensions. These properties are frequently used to
establish results in subsequent sections. The first three parts are proved in \cite{adh99},
the fourth result is evident using symmetry. In the Appendix, we state and prove a simple
extension to the convexity property (see Lemma \ref{refconvex}). The convexity property below implies that direct
paths within a face should have constant velocity and direction. The
extension shows that this also is true for reflected paths.

\begin{lemma}  \label{lem:opp}

\vspace{0.1in}

\begin{enumerate}
\item (Optimality of Linear Paths via Convexity) Let $g$ be a convex function on $\mathbb{R}^{d}$, and $x \in \mathcal{H}^{d}$. Then for $t_{1} < t_{2}$, $$\int^{t_{2}}_{t_{1}} g(\dot{x}(t))dt \geq \int^{t_{2}}_{t_{1}} g\left(\frac{x(t_{2})-x(t_{1})}{t_{2}-t_{1}}\right) dt.$$
    This implies that a (one-piece) linear path minimizes the unconstrained variational problem.
\item (Scaling) Consider a variational problem with $v \in \mathbb{R}^{d}_{+}$. For $\forall k > 0$, $I(kv) = kI(v)$. Furthermore, if $(x,y,z)$ is the optimal triple for $v$ and $\hat{x},\hat{y},\hat{z}$ is the optimal triple for $kv$, then $\hat{x}(t)=kx(t/k)$, $\hat{y}(t)=ky(t/k)$, $\hat{z}(t)=kz(t/k)$.
\item (Merging Paths) Let $(x_{1}, y_{1}, z_{1})$ be an $R$-regulation triple on $[0,t_{1}]$ with $z_{1}(0) = 0$ and $z_{1}(t_{1}) = w$ and $(x_{2},y_{2},z_{2})$ be an optimal triple on $[s_{2},t_{2}]$ with $z_{2}(s_{2}) = w$ and $z_{2}(t_{2}) = v$. Suppose both $x_{1}$ and $x_{2}$ are absolutely continuous. Define
    $$z(t) = \left \{ \begin{array}{ll}
    z_{1}(t)            & \mbox{for } 0 \leq t \leq t_{1},  \\
    z_{2}(t-t_{1}+s_{2}) & \mbox{for } t_{1} \leq t \leq t_{1}+t_{2}-s_{2}, \\
                       \end{array} \right.
                       $$
    $$x(t) = \left \{ \begin{array}{ll}
    x_{1}(t)            & \mbox{for } 0 \leq t \leq t_{1},  \\
    x_{2}(t-t_{1}+s_{2}) & \mbox{for } t_{1} \leq t \leq t_{1}+t_{2}-s_{2}, \\
                       \end{array} \right.
                        $$
    $$y(t) = \left \{ \begin{array}{ll}
    y_{1}(t)            & \mbox{for } 0 \leq t \leq t_{1},  \\
    y_{2}(t-t_{1}+s_{2}) & \mbox{for } t_{1} \leq t \leq t_{1}+t_{2}-s_{2}, \\
                       \end{array} \right.
                        $$
and $s=t_{1}+t_{2}-s_{2}$. Then $(x,y,z)$ is an $R$-regulation triple on $[0,s]$ with $z(0) = 0$ and $z(s) = v$.

\item (Symmetric Terminal Points) Consider an RSVP. If $v_{1} = (a,b,c)$, $v_{2} = (b,c,a)$, $v_{3} = (c,a,b)$
with $a,b,c \ge 0$, then $I(v_{1}) = I(v_{2}) = I(v_{3})$. Furthermore, each optimal path to one of these points is
rotationally symmetric translation of an optimal path to one of the other points.
For an MSVP case, let $v'_{1} = (b,a,c)$, $v'_{2} = (c,b,a)$, $v'_{3} = (a,c,b)$, then $I(v_{1}) = I(v_{2}) = I(v_{3}) = I(v'_{1}) = I(v'_{2}) = I(v'_{3}).$
\end{enumerate}
\end{lemma}

\section{Eliminating Bad Faces in RSVPs}

The next result is one of the key results in the paper, since it allows us to eliminate
entire categories of paths by eliminating paths whose penultimate pivot point is on a ``bad face.''
Below, the distance between a face and a point is the standard Euclidean distance from
a point to the associated face.
\begin{theorem}[Bad Faces] \label {theorem:shortest}
Consider a variational problem with terminal point $v \in int(\reals^3_+)$
and consider an optimal triple $(x,y,z)$ to $v$. Let $w$ be the last point of
$z$ which is not in $int(\reals^3_+)$. Then
\begin{itemize}
  \item[(a)] If the variational problem is an RSVP, then there exists an optimal path for which $w$ is in
one of the two nearest faces to $v$.
  \item[(b)] If the variational problem is an MSVP, then there exists an optimal path for which $w$ is in
  the nearest face to $v$.
\end{itemize}
\end{theorem}
\begin{proof} Let the terminal point be $v = (v_{1}, v_{2}, v_{3})$,
and without loss of generality,
assume $v_3 \ge v_2 \ge v_1 > 0$.
Define $u^{1}=(0,a,b)$, $u^{2}=(b,0,a)$, $u^{3}=(a,b,0)$ and $u^{4}=(0,b,a)$.
Furthermore, we assume that we cannot have both $a=0$ and $b=0$.

First, we want to compare the optimal cost from $u^i$ to $v$ for various
values of $i$. Of course, by convexity (Lemma \ref{lem:opp}, part 1), the optimal path from any $u^i$ to
$v$ must be a linear path. Now, recall that
$$\tilde{\mathcal{I}}_{0}(u^{i},v) = \|\theta\|\|v-u^{i}\|-\langle\theta,v-u^{i}\rangle,$$
for $i \in \{1, 2, 3, 4\}$.
It can be checked that
$$\langle\theta,v-u^{i}\rangle = \theta_{0}\sigma^{-2}(2\gamma_{1}+\gamma_{0})(v_{1}+v_{2}+v_{3}-a-b).$$
Hence this portion of the cost is independent of $i$.
So it is sufficient to analyze $\|v-u^{i}\|$ or, equivalently, $\|v-u^{i}\|^{2}$.
Note then that
$$\|v-u^{i}\|^{2} = \langle v,v\rangle + \langle u^{i},u^{i}\rangle - 2\langle v,u^{i}\rangle,$$
where $\langle u^{i},u^{i}\rangle = \sigma^{-2}[(a^{2}+b^{2})\gamma_{0}+2ab\gamma_{1}]$.
Therefore, the first two terms in $\|v-u^{i}\|^{2}$ are also independent of $i$.
So, finally, we have
$$\langle v, u^{1}-u^{3}\rangle = \sigma^{-2}[(\gamma_{0}-\gamma_{1})(v_{3}-v_{2})b+(\gamma_{0}-\gamma_{1})(v_{2}-v_{1})a] \geq 0,$$
where the inequality is due to our assumption on $v$ and Lemma \ref{lemma:matrix}.
This of course implies that $\langle v, u^{1} \rangle \ge \langle v, u^{3}\rangle$ with
equality iff $(v_{3}-v_{2})b+(v_{2}-v_{1})a = 0$. Therefore we have
\begin{equation} \label{eqn:mus}
\tilde{\mathcal{I}}_{0}(u^{3},v) \ge \tilde{\mathcal{I}}_{0}(u^{1},v).
\end{equation}
Now, let $w$ be the last point which is not in the interior of the octant, for an optimal path with
terminal point $v$.
By convexity, the segment $\overline{wv}$ must be linear.
Suppose then that $w$ is in $F_3$. If $v_3 \not= v_2$ then $F_3$ is
the furthest face from $v$. Thus $w$ must be of the form of
$u^3$ and accordingly we take $w=u^3=(a,b,0)$.
Now, consider the point $u^1=(0,a,b)$.
Note that the optimal cost from the origin to $u^3$ and the optimal cost from the  origin to $u^1$ must be equal due to rotational symmetry. By the merging and convexity properties of Lemma \ref{lem:opp}, to establish (a) it suffices to show that the
optimal cost from $u^1$ to $v$ using a direct path is less than or equal the cost from $u^3$ to $v$
via a direct path. This result was already demonstrated, as seen in (\ref{eqn:mus}).

At this point some discussion may be needed to see that part (a) of the theorem has been proved.
Suppose first that $v_3 > v_2 > v_1 > 0$. Then $F_3$ is the unique furthest face from
$v$. In this case we can strengthen the conclusion of part (a). In particular, the last boundary
point in an optimal path must emanate from one of the two nearest faces. Next, if $v_2=v_3$ then
all three faces can be classified as ``one of the two nearest'' and the statement of (a) holds
by default. Finally, if $v_3 > v_2=v_1 > 0$ there are two cases. If $b \not=0$, then $u^3$
is in the interior of $F_3$ and there must exist a strictly cheaper path through $u^1$.
If $b=0$, then the cost of the paths through $u^3$ and $u^1$ are the same. Either path is
considered to be via $F_1$(albeit on the boundary) which is one of the two nearest faces to $v$.
Hence, the result in (a) is still valid.

We now address part (b) of the theorem. First, it can be checked that
$$\langle v, u^{4}-u^{2}\rangle = \sigma^{-2}[b(v_{2}-v_{1})(\gamma_{0}-\gamma_{1})] \geq 0,$$
with equality if $b=0$ or $v_1=v_2$. Using the calculations from the RSVP case, we have
\begin{equation} \label{eqn:mus2}
\tilde{\mathcal{I}}_{0}(u^{2},v) \ge \tilde{\mathcal{I}}_{0}(u^{4},v).
\end{equation}
By mirror symmetry, the optimal cost from the origin to $u^3$ and the optimal cost from the  origin to $u^4$ must be equal.

The remainder of the proof is similar to the part (a) argument. Again, let $w$ be the last point which
is not in the interior of the octant, for an optimal path with terminal point $v$.
Suppose first that $w$ is in $F_3$. Unless $v_1=v_2=v_3$ (in which case the result holds trivially),
then $F_3$ is one of the two furthest faces from $v$. Recall that an MSVP is also an RSVP, so we can
apply part (a) of the theorem to conclude that there must exist an optimal path to $v$ with $w \in F_2$.
Without loss of generality, assume then that $w=u^2=(b,0,a)$.
By mirror symmetry, the optimal cost from the origin to $u^2$ and the optimal cost from the  origin to $u^4$ must be equal.
However, by (\ref{eqn:mus2}) the
optimal cost from $u^4$ to $v$ using a direct path is less than or equal the cost from $u^2$ to $v$
via a direct path. Hence, there exists a path for which $u^4$ is last point on the boundary of the octant,
with lower (or equal) cost to the path through $u^2$.

As in part (a), there are some special cases in part (b), specifically, when $v_1=v_2$ or $b=0$.
If $v_1=v_2$ then both $F_1$ and $F_2$ are considered the ``nearest face'' and the statement
holds immediately by applying part (a). If $v_2 > v_1$ and $b \not=0$, then there exists a strictly
cheaper path through $u^4$. If $v_2 > v_1$ and $b=0$, then $u^2$ and $u^4$ coincide and they are considered to be in
$F_1$, immediately implying the result.
\end{proof}

The easiest way of rephrasing the RSVP result is as follows. Consider a terminal point $v$ with
a unique farthest face. Then the last linear segment in an optimal path cannot emanate from the
interior of the farthest face. Similarly, for an MSVP with a unique nearest face to the terminal
point $v$, the last linear segment must emanate from the nearest face.

Note that the results in Theorem \ref{theorem:shortest} are proved only for $v$ in the interior of the octant.
The arguments in the proof of the theorem lead immediately to the following extensions for terminal points on the boundary of the octant.

\begin{remark} When $v_{1} = 0$ and $v_{2} > 0$, $v$ is in the interior of $F_1$. The first part of Theorem \ref{theorem:shortest} holds in the following sense: For an RSVP, there exists an optimal path whose last segment does not emanate from the interior of $F_3$. Furthermore, the last segment can not originate from $F_{2,3}$ although
it may originate from $F_{1,3}$.
\end{remark}

\begin{remark} When $v_{1} = v_{2} =0$ and $v_{3} > 0$, $v$ is on the axis $F_{1,2}$.
Again, the first part of Theorem \ref{theorem:shortest} holds.
In particular, there exists an optimal path whose last segment does not emanate from the interior
of the farthest face, which is $F_3$ in this case.
\end{remark}

Finally, we believe that the results of this section can be generalized to higher dimensional RSVPs and MSVPs
with minor modifications to the proofs.

\section{Further Optimal Path Characterizations} \label{sec:finite}

Our eventual goal is to show that optimal paths in three dimensions can be of only two types:
gradual paths and classic spirals. Demonstrating this requires the establishment of a number of
properties for paths with a finite or infinite number of linear segments. The results in this section
are related to paths with a finite number of segments, although some of these properties are used
later on to establish characterizations for paths with an infinite number of segments.

In various proofs in this section it is useful to consider paths (and the corresponding costs),
which are feasible, but not necessarily optimal. Therefore, we introduce the following definition.
\begin{definition} (Cost of a Feasible Path)
Given a one-piece feasible R-regulated triple $(x,y,z)$
with $z(0)=u$ and $z(T)=v$, define the corresponding cost along that path to be
$$H_{x}(u,v)= \frac{1}{2} \int^{T}_{0} \parallel \dot{x}(t)-\theta \parallel ^{2} 
dt.$$
Note that the pair $(y,z)$ uniquely defines $x$. 
\end{definition}

The next several results provide detailed characterizations of optimal paths. Unfortunately,
the overall connection will not be apparent until we bring them together to prove the main
results.

\begin{lemma}[The Switchback Lemma]   \label{lem:switchback}
Consider a VP with $\Gamma = I$. Let $v^{1}, v^{4} \in int(F_{1})$ and
$v^{2}, v^{3} \in int(F_{2})$. Then the path from $v^1$ to $v^4$ consisting
of the following linear segments is
strictly suboptimal: a direct segment from $v^1$ to $v^2$, a reflected segment from $v^2$
to $v^3$, a direct segment from $v^3$ to $v^4$.
\end{lemma}
\begin{proof}
Let $v^{1} = ( 0,v^{1}_{2}, v^{1}_{3})$, $v^{2} = (v^{2}_{1}, 0, v^{2}_{3})$, $v^{3} = (v^{3}_{1}, 0, v^{3}_{3})$, and  $v^{4} = (0,v^{4}_{2}, v^{4}_{3})$.  Assume first that $v^{2}_{1} \geq v^{3}_{1}.$
Define $\tilde{v}^{2} = (v^{2}_{1}-v^{3}_{1}, 0, v^{2}_{3})$ and $\tilde{v}^{3} = (0, 0, v^{3}_{3})$ which are both in $F_2$. Consider a new path from $v^1$ to
$v^4$ as follows: a direct segment from $v^1$ to $\tilde{v}^2$, a reflected segment
from $\tilde{v}^2$
to $\tilde{v}^3$, a direct segment from $\tilde{v}^3$ to $v^4$. We show that
the new path has a strictly lower cost than the original path.
 Notice that $v^{3}-v^{2} = \tilde{v}^{3} - \tilde{v}^{2}$ so it suffices to compare $\tilde{\mathcal{I}}_{0}(v^{1},v^2) + \tilde{\mathcal{I}}_{0}(v^{3},v^4)$ with  $\tilde{\mathcal{I}}_{0}(v^{1},\tilde{v}^2) + \tilde{\mathcal{I}}_{0}(\tilde{v}^{3},v^4)$.
By definition $$\tilde{\mathcal{I}}_{0}(v^{1},v^2) + \tilde{\mathcal{I}}_{0}(v^{3},v^4) = \| \theta \| ( \|v^{2}-v^{1}\| + \|v^{4}-v^{3}\|) - \langle \theta, v^{2}-v^{1}+v^{4}-v^{3} \rangle$$ and $$\tilde{\mathcal{I}}_{0}(v^{1}, \tilde{v}^2) + \tilde{\mathcal{I}}_{0}(\tilde{v}^{3},v^4) = \| \theta \| ( \| \tilde{v}^{2}-v^{1} \| + \|v^{4}-\tilde{v}^{3}\|) - \langle \theta, \tilde{v}^{2}-v^{1}+v^{4}-\tilde{v}^{3} \rangle.$$
It is easy to check that $$\langle \theta, v^{2}-v^{1}+v^{4}-v^{3} \rangle = \langle \theta, \tilde{v}^{2}-v^{1}+v^{4}-\tilde{v}^{3} \rangle$$ so it is enough to compare $\|v^{2}-v^{1}\| + \|v^{4}-v^{3}\|$ with $\|\tilde{v}^{2}-v^{1} \| + \|v^{4}-\tilde{v}^{3}\|$. Now when
$\Gamma=I$, we have $$(\|v^{2}-v^{1}\| + \|v^{4}-v^{3}\|) - (\|\tilde{v}^{2}-v^{1} \| + \|v^{4}-\tilde{v}^{3}\|) = \sqrt{p+(v^{2}_{1})^{2}} + \sqrt{q+(v^{3}_{1})^{2}}-\sqrt{p+(v^{2}_{1}-v^{3}_{1})^{2}}-\sqrt{q} > 0$$ where $p = (v^{1}_{2})^{2} + (v^{2}_{3}-v^{1}_{3})^{2} > 0$ and $q = (v^{4}_{2})^{2}+(v^{4}_{3}-v^{3}_{3})^{2} > 0$.
Thus, the newly constructed path has a strictly lower cost.
If $v^{2}_{1} < v^{3}_{1}$ then re-define $\tilde{v}^{2} = (0, 0, v^{2}_{3})$ and $\tilde{v}^{3} = (v^{3}_{1}-v^{2}_{1}, 0, v^{3}_{3})$. The proof of the corresponding result for this case is analogous to the first case.
\end{proof}
This result is the key to showing that ``exotic'' paths which seem intuitively ``bad''
are indeed suboptimal. In particular, it shows that paths which switch back and
forth between two faces are not cost effective. Analogous arguments show that
the lemma holds for any pair of two-dimensional faces.

The next result is important in establishing the optimality of gradual paths.
In this and later proofs, we use the standard notation $e_3 = (0,0,1)$.

\begin{lemma} \label{axis_eliminate}
Consider an RSVP with $\Gamma=I$ and $r_{2} \geq 0$. Let $v^{1} = (0, v^{1}_{2}, v^{1}_{3})$ and $v^{2} = (v^{2}_{1}, 0, v^{2}_{3})$, such that   $v^1 \in int(F_{1})$ and $v^2 \in  int(F_{2})$. Then the path from $v^1$ to $e_3$ consisting
of the following linear segments is
strictly suboptimal: a direct segment from $v^1$ to $v^2$ and a reflected segment from $v^2$
to $e_3$.
\end{lemma}

\begin{proof}
Define $\tilde{v}^{2} = (0, 0, v^{2}_{3})$.
We show that
$$\tilde{\mathcal{I}}_{0}(v^{1},v^2) + \tilde{\mathcal{I}}_{2}(v^{2},e_3) >
\tilde{\mathcal{I}}_{0}(v^{1},\tilde{v}^{2}) + \tilde{\mathcal{I}}_{2}(\tilde{v}^{2}, e_3),$$
implying that there exists a better path from $v^1$ to $e_3$, via $\tilde{v}^2$.

Suppose $(x^{1}, y^{1}, z^{1})$ is an optimal triple from $v^2$ to $e_3$ with
corresponding time $T^1$, and let $(x^{2}, y^{2}, z^{2})$ be the optimal triple
from $v^1$ to $v^2$ with
corresponding time $T^2$ (since this path is direct $x^2=z^2$).
Set $(x^{1}(t), y^{1}(t), z^{1}(t)) = (\dot{x}^1, \dot{y}^1, \dot{z}^1)t$, $x^{2}(t) = \dot{x}^{2}t$, with $\dot{z}^1 = (z^{1}_{1}, z^{1}_{2}, z^{1}_{3})'$. Notice that $\dot{y}^{1}= (0, y^{1}_{2},0)'$ and $z^{1}_{2} = 0$. Let $\tilde{z}^{1}(t) = t(0, 0, z^{1}_{3})'$ and $\tilde{x}^{1}(t) = t(-r_2 y^{1}_{2}, -y^1_2, x^{1}_{3})'$. It can be checked that $(\tilde{x}^{1}, y^{1}, \tilde{z}^{1})$ is a feasible
triple from $\tilde{v}^2$ to $e_3$ with $\tilde{T}^1 = T^1$.
Similarly, setting $\tilde{x}^{2}(t) = t(0, x^{2}_{2}, x^{2}_{3})'$ yields the feasible triple
$(\tilde{x}^{2}, 0, \tilde{x}^{2})$ from $v^1$ to $\tilde{v}^2$, with $\tilde{T}^2 = T^2$.

Using the paths defined above we have:
 $$\tilde{\mathcal{I}}_{2}(v^{2}, e_3) = \frac{1}{2} T^{1}[(z^{1}_{1}-r_{2}y^{1}_{2}-\theta_{0})^{2} + (-y^{1}_{2}-\theta_{0})^{2} + (z^{1}_{3}-r_{1}y^{1}_{2}-\theta_{0})^{2}]$$ and $$\tilde{\mathcal{I}}_{2}(\tilde{v}^{2}, e_3) \leq \frac{1}{2} T^{1}[(-r_{2}y^{1}_{2}-\theta_{0})^{2} + (-y^{1}_{2}-\theta_{0})^{2} + (z^{1}_{3}-r_{1}y^{1}_{2}-\theta_{0})^{2}],$$
where the inequality is due to the fact that $(\tilde{x}^{1}, y^{1}, \tilde{z}^{1})$ need not be optimal.

Recall that $\dot{z}^{1}T^{1} = e_{3}-v^{2}.$ Finally, we expand the direct path costs similarly
and compute:
\begin{align*}
&  \hspace{-0.5in}  \tilde{\mathcal{I}}_{0}(v^{1},v^2) + \tilde{\mathcal{I}}_{2}(v^{2}, e_3) - \tilde{\mathcal{I}}_{0}(v^{1},\tilde{v}^{2}) - \tilde{\mathcal{I}}_{2}(\tilde{v}^{2}, e_3) \\
\geq \, &\frac{1}{2}T^{1}\left[\left(-\frac{v^{2}_{1}}{T^{1}}-r_{2}y^{1}_{2}-\theta_{0}\right)^{2} - (-r_{2}y^{1}_{2}-\theta_{0})^{2}\right] + \frac{1}{2}T^{2}\left[\left(\frac{v^{2}_{1}}{T^{2}}-\theta_{0}\right)^{2} - (-\theta_{0})^2\right] \\
= &  \;  \frac{1}{2}\left[\frac{(v^{2}_{1})^{2}}{T^{1}}+\frac{(v^{2}_{1})^{2}}{T^{2}}+2r_{2}v^{2}_{1}y^1_2
\right].
\end{align*}
Since $r_2 \ge 0$, $y^1_2 \ge 0$, and $v_1^2$ can assumed to be positive,
the last term is strictly positive, establishing the result. (If $v_1^2=0$ then the theorem
holds trivially by convexity.)
\end{proof}
In general, we apply this result under the condition that $r_{1}, r_{2} \geq 0$. By symmetry it is easily seen
that the result applies to rotational variations of the paths involved in the result.

\begin{lemma} \label{differentR}
Consider an RSVP and the points $v=(0,v_{1},v_{2})$ and $\bar{v}=(v_{1},0,v_{2})$.
If $r_{1} \geq r_{2}$ and $v_{2} \geq v_{1}$, then $$\tilde{\mathcal{I}}_1(v) \geq \tilde{\mathcal{I}}_2(\bar{v}).$$
\end{lemma}
\begin{proof}
Let $(x^{*}(t),y^{*}(t),z^{*}(t))$ be an optimal triple corresponding to $\tilde{\mathcal{I}}_1(v)$,
and let $T^*$ be the corresponding optimal time. It is clear that $\dot{z}^{*}(t)$ and $\dot{y}^{*}(t)$
are constant functions due to the convexity property of Lemma \ref{lem:opp}. Thus, we set
$z^*: = \dot{z}^{*}(t)=(0,z_{1}^{*}, z_{2}^{*})'$ and $y^*=\dot{y}^{*}(t)=(y_{1}^{*},0,0)'$. Therefore, $\dot{x}^{*}(t)=\dot{z}^{*}-R\dot{y}^{*}$ and $z^{*}$ satisfies $z_{1}^{*}T^{*} = v_{1}$, and $z_{2}^{*}T^{*} = v_{2}$. It is clear that $z_{1}^{*} \leq z_{2}^{*}$ since $v_{1} \leq v_{2}$.

Setting $\dot{\bar{z}}(t)=\bar{z}=(z_{1}^{*}, 0, z_{2}^{*})'$, $\dot{\bar{y}}(t)=\bar{y}=(0,y_{1}^{*},0)'$ and $\dot{\bar{x}}(t)=\dot{\bar{z}}-R \dot{\bar{y}}$, we note that $(\bar{x}(t),\bar{y}(t),\bar{z}(t))$ is
a feasible one-piece triple from the origin to $\bar{v}$ where $\bar{z}(t) \in F_2$ for $t \ge 0$ and $T^*$ is the corresponding time for this path to reach $\bar{v}$.
Then $$\tilde{\mathcal{I}}_1(v)=\frac{1}{2} \|\dot{z}^{*}(t)-R \dot{y}^{*}(t)-\theta\|^{2}T^{*}=\frac{1}{2} \|z^{*}-R y^{*}-\theta\|^{2}T^{*}.$$
On the other hand, since $(\bar{x}(t),\bar{y}(t),\bar{z}(t))$ is feasible,
$$\tilde{\mathcal{I}}_2(\bar{v}) \leq H_{x}(\bar{v}) = \frac{1}{2}\|\bar{z}-R \bar{y}-\theta\|^{2}T^{*}.$$
So $$\tilde{\mathcal{I}}_1(v) - \tilde{\mathcal{I}}_2(\bar{v}) \geq \frac{1}{2}(\|z^{*}-Ry^{*}-\theta\|^{2} - \|\bar{z}-R \bar{y}-\theta\|^{2})T^{*} = \sigma^{-2}(r_{1} - r_{2})(\gamma_{0} - \gamma_{1})(z_{2}^{*}-z_{1}^{*})y_{1}^{*}T^{*} \geq 0.$$
The last inequality follows from our assumptions and because $\gamma_0 > \gamma_1$, due to Lemma \ref{lemma:matrix}.
\end{proof}

Our study of optimal path characterizations rests crucially on comparing the paths depicted in
Figure \ref{fig:redblue}. For paths with both a finite number of segments and an infinite number of segments,
we wish to establish that the blue path is ``cheaper'' than the red path. In most instances
it is difficult to establish this as a general property, so we provide a sufficient condition
for this blue-path-red-path condition to hold. For specific numerical instances of an RSVP,
this condition is easily verified. Furthermore, a combination of numerical
and analytical arguments can be used to show that the condition holds in general on $R_f$, defined
below.

\begin{figure}
\begin{centering}
\includegraphics[width=8cm]{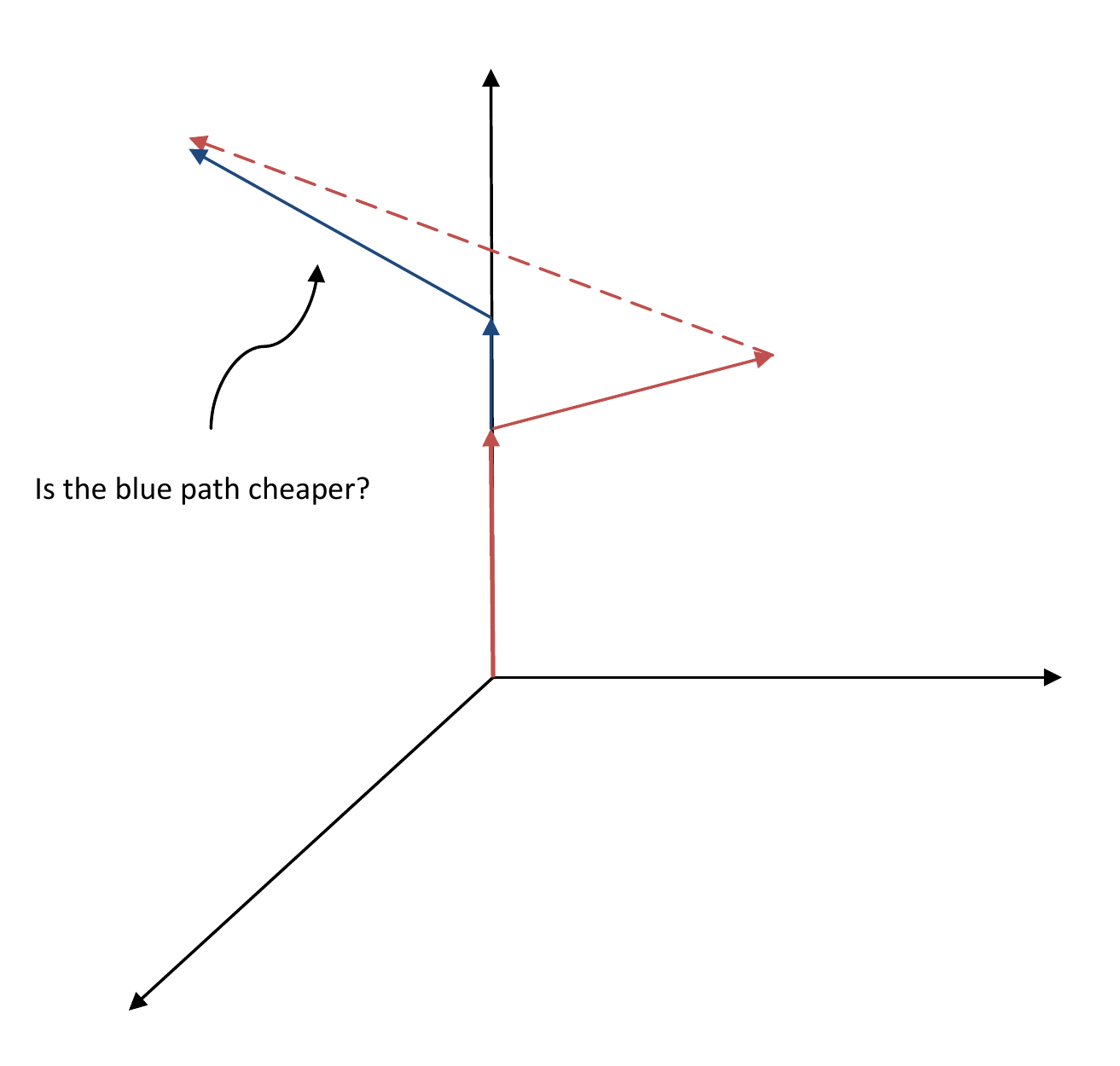}\\
\caption{Red Path - Blue Path Comparison} \label{fig:redblue}
\end{centering}
\end{figure}

Let $R_f = \{ (r_1,r_2) \in \R^2_+ \; | \: r_1 > r_2, -1 < r_1 + r_2< 2 \}$.

\begin{condition} For an RSVP with reflection matrix $R$,
$(r_1,r_2) \in R_f$ and
$$(1+r_{2}^2)(1+r_{1}^2-r_{2}-r_{1}r_{2})^{2} \geq 2(r_{1}r_{2})^{2}(1+r_{1}^{2}+r_{2}^{2}-r_{1}-r_{2}-r_{1}r_{2}).$$
\end{condition}

This condition is required to prove the next two results. Lemma \ref{lem:case3} is proved
in the Appendix.

\begin{lemma} \label{lem:case3}
Given an RSVP with $\Gamma=I$, $\theta_0 < 0$, and $r_1 > r_2 \ge 0$, define points $v=(v_{1},0,\bar{v}_{3})$
and $v'=(0,v_{2},v_{3})$ with $\bar{v}_3 > 0$ and $v_i > 0$ for $i=1,2,3$.
Suppose further that $v_{2} < v_{3}$  and $v_1 < \bar{v}_3$. Then
\begin{itemize}
  \item[(a)] If Condition 1 holds, then for all $k \in [0,1],$  $\tilde{\mathcal{I}}_2(v) > \tilde{\mathcal{I}}_2(k \bar{v}_3e_{3})$. 
  \item[(b)] There exists a $k \in [0,1]$ such that 
  $ \tilde{\mathcal{I}}_{0}(v,v') \geq \tilde{\mathcal{I}}_1(k \bar{v}_3e_{3},v').$
\end{itemize}
\end{lemma}

\begin{lemma} \label{DFO}
Given an RSVP with $\Gamma=I$, $\theta_0 < 0$, and $r_1,r_2 \ge 0$, define points $v=(v_{1},0,\bar{v}_{3})$
and $v'=(0,v_{2},v_{3})$ with $\bar{v}_3 >0$ and $v_i > 0$ for $i=1,2,3$. Then if Condition 1 holds, the least cost
two-piece path from the origin to $v$ and from $v$ to $v'$ is not an optimal
path to $v'$.
\end{lemma}

\begin{proof} When $v_{2} \geq v_{3}$, the two-piece path through $v$ is not optimal due to Theorem \ref{theorem:shortest}. Next, invoking the scaling property of Lemma \ref{lem:opp}, we assume without loss of generality that
$\bar{v}_3=1$. 

\textbf{Case 1.} Consider then the case where $v_2<v_3$, $v_1 < 1$, $r_{2} \geq r_{1}$ and set $\hat{v}=(0,v_{1},1)$.
We claim that the optimal two-piece path through $\hat{v}$ is strictly better than
the optimal two-piece path through $v$. First, Lemma \ref{differentR} gives $\tilde{\mathcal{I}}_2(v) \geq \tilde{\mathcal{I}}_1(\hat{v})$. In other words, the first segment of the path through $\hat{v}$ has
a lower (or equal) cost than the first segment through $v'$.

To compare the second segments, note that
\begin{eqnarray*}
\tilde{\mathcal{I}}_{0}(v,v') & = & \| \theta \| \|v'-v \| - \langle \theta, v'-v \rangle \qquad \mbox{and} \\ \tilde{\mathcal{I}}_{0}(\hat{v},v') & = & \| \theta \| \|v'-\hat{v}\| - \langle \theta, v'-\hat{v}\rangle.
\end{eqnarray*}
Furthermore, we have $\langle \theta, v'-v \rangle = \langle \theta, v'-\hat{v}\rangle,$ and  $$ \|v'-v \| = \sqrt{v_{1}^{2} +v_{2}^{2}}> \sqrt{(v_{2}-v_{1})^{2}} = \|v'-\hat{v} \|.$$
Thus, $\tilde{\mathcal{I}}_{0}(v,v') > \tilde{\mathcal{I}}_{0}(\hat{v},v')$ and the result is established
for this case.

\textbf{Case 2.} Suppose next that $v_2<v_3$ and $v_1 \ge 1$ (with no restriction on $r_1$ and $r_2$).
Let $\tilde{v} = (0,1,v_{1})$ and consider the two-piece path to $v'$ via $\tilde{v}$. Again,
we show that the optimal two-piece path through $\tilde{v}$ is strictly better than
the optimal two-piece path through $v$. By rotational symmetry $\tilde{\mathcal{I}}_2(v) = \tilde{\mathcal{I}}_1(\tilde{v})$. On the other hand $\langle \theta, v'-v \rangle = \langle \theta, v'-\tilde{v} \rangle$, and $$ \|v'-v \| = \sqrt{v_{1}^{2} +v_{2}^{2}}> \sqrt{(v_{2}-1)^{2}+(1-v_{1})^{2}} = \|v'-\tilde{v} \|.$$ As in Case 1, this implies $$\tilde{\mathcal{I}}_{0}(v,v') > \tilde{\mathcal{I}}_{0}(\tilde{v},v'),$$
which establishes the result for this case.

\textbf{Case 3.} The remaining case is when $v_{2} < v_{3}$, $r_{2} < r_{1}$, and $v_1 < 1$.
Once again we find an alternate two-piece path to $v'$ which has a lower cost. In this
case, consider the two-piece path to $v'$ via $e_3$. If Condition
1 holds, then Lemma \ref{lem:case3} indicates that for some $k \in [0,1]$, 
$\tilde{\mathcal{I}}_2(v) > \tilde{\mathcal{I}}_2(ke_{3})$ and $ \tilde{\mathcal{I}}_{0}(v,v')
\geq \tilde{\mathcal{I}}_1(ke_{3},v')$.
This establishes the result for this case.
\end{proof}

Note that Condition 1 is only needed to establish the third case. It may be possible
to replace this condition by a simpler expression for special cases.



Next, we are now able to establish the result that there always exist
gradual optimal paths to points on the boundary of the octant. It is important
to note that the class of gradual paths do not include paths which traverse an
axis and then cross the interior to a point on a two-dimensional face.
\begin{theorem} \label{thm:gradual}
Consider an RSVP with $\Gamma = I$, $\theta_0 < 0$, $r_1,r_2 \ge 0$. Suppose Condition 1 holds
and that there exists an optimal path with a finite number of segments.
Then:
\begin{itemize}
\item[(i)] For any point on an axis there exists an optimal path consisting of
a single segment; and
\item[(ii)] For any point on a two-dimensional face there exists an optimal
gradual path, consisting of one or two segments.
\end{itemize}
\end{theorem}

\begin{proof}
To prove the result we need to eliminate a large number of path types. In order to
categorize these types, note that each type can be classified according to the endpoints of the
linear segments. The endpoint of each piece can be on the interior of a two-dimensional face ($F$),
on the interior of an axis ($A$), or the origin ($O$). In all the arguments below, we consider a path
with a finite number of pieces, and thus a finite number of endpoints, which starts at a point
$v$ in the octant and terminates at the origin. Specifically, we label the endpoints in ``reverse
order.'' Note that an endpoint cannot be in the interior
of the octant due to the convexity property in Lemma \ref{lem:opp}.
Furthermore, note that the last point is always of type $O$ and of course, this is the only
position at which this type occurs.

Next, consider an endpoint of type $F$. There are three possibilities for the
previous endpoint:
\begin{itemize}
\item The endpoint is on an axis $A$ (either one of the two axes adjoining this face, or
the remaining axis)
\item The endpoint is on the same face $SF$.
\item The point is on a difference face $DF$.
\end{itemize}
Similarly, for an endpoint of type $A$, there are two possibilities for the
previous endpoint:
\begin{itemize}
\item The point is on an axis $A$.
\item The point is on the same face $SF$ (i.e., a face adjoining the axis).
\end{itemize}
Notice that for a point of type $A$ the previous point cannot be on the face not adjoining the axis
as a consequence of the Bad Faces Theorem.
With this notation, we can categorize a piecewise linear path by a finite series whose
elements are in the set $\{SF, DF, A, O \}$.

For a series corresponding to a finite-piece optimal path, we infer the following rules:
\begin{enumerate}
\item By the convexity property of
optimal paths, none of the following pairs can appear in the series: $(DF, DF)$, $(SF,SF)$, $(SF, A)$.
\item If $A$ appears somewhere in the series, then the end of the series cannot be $(A,O)$, due to the scaling
and symmetry properties of optimal paths. The only exception is, of course a series which is simply $(A,O)$.
\item The series cannot end with $(SF, O)$ by convexity.
\item The series cannot end with $(DF, O)$ due to Lemma \ref{DFO}.
\end{enumerate}
Note that rules 3 and 4 imply that the series must end with $(A,O)$.

We now establish
part (i). Consider a path with the terminal point on say axis $F_{1,2}$
and the first segment of the path emanating from the origin. If this first
segment traverses an axis, then by scaling and symmetry, part (i) immediately
holds for any terminal point on an axis. By convexity, the first segment cannot
be in the interior of the octant. So, the first segment must be embedded in
a two-dimensional face. Now, the second segment cannot be embedded
in this same face due to convexity. So, it must cross the interior and
terminate either in a different face, or on the opposing axis. The first
case is ruled out by Lemma \ref{DFO}. The second case is not possible
by the Bad Faces theorem. Hence, the first, and only segment, must be embedded
in an axis.

In consideration of Rules 1 through 4 above, to prove part (ii) we must exclude two
remaining cases:
 $(F, DF, A, O)$ and $(F, DSF_{i},DF, A, O)$, $i=0,1,2, \ldots.$  Here $SDF_{i}$ is a subsequence of $(SF,DF)$ that repeats $i$ times, and $DSF_{i}$ is a subsequence of $(DF,SF)$ that repeats $i$ times. Consider the case $(F, DF, A, O)$ first.
Without loss of generality let the terminal point be on face $F_1$, denote it
 $v^{1}=(0,v^{1}_{2}, v^{1}_{3})$, and assume that $0 < v^{1}_{2} \leq v^{1}_{3}$.
The endpoint before $v^{1}$ has to be a point $v^{2}$ on $F_2$. So $v^{2}=(v^{2}_{1},0, v^{2}_{3})$ which is the $DF$ in the series. We must have $v^{2}_{1} \leq v^{2}_{3}$
for the path to be optimal, by the assumed type of path and the Bad Faces Theorem.
The next endpoint
$v^3$ cannot be on axis $F_{1,3}$ again by the Bad Faces Theorem. Furthermore, it cannot be on axis $F_{1,2}$ due to the arguments in the proof of Lemma \ref{DFO}. Hence,
$v^3$ must be in $F_{2,3}$.

Now, if the path just described is optimal, this implies that the optimal path to $v^{2}$ is from the origin to $v^{3}$ then to $v^{2}$.
Then by the scaling property, the optimal path to an arbitrary point $(u_{1}, 0, u_{3})$
on face $F_2$ is of the form $(F,A,O)$ if $u_{1} \leq u_{3}$.
By symmetry, the optimal path to an arbitrary point $(0, u_{2}, u_{3})$ on $F_1$ is also of the form $(F,A,O)$ if $u_{2} \leq u_{3}$. However, $v^{1}$ is indeed of this form, which
means we can replace the proposed optimal path of the form $(F,DF, A, O)$
by a gradual path of the form $(F,A,O)$.

Next consider the case $(F, SF, DF, A, O)$. If the terminal point $v$ is in say $F_1$ then
so is the endpoint $v^1$ immediately preceding this point. This implies that the optimal
path to $v^1$ is of the form $(F, DF,A,O)$. As argued above, we can eliminate this
form. All the remaining cases can be eliminated by analogous arguments that reduce the
end of the series to the $(F,DF,A,O)$ case. We conclude that any optimal path to
a point on the interior of a two-dimensional face can be reduced to the gradual
path forms $(F,A,O)$ or $(F,O)$. This establishes part (ii).
\end{proof}

Finally, we present the main result for optimal paths with a finite number of segments.

\begin{theorem} \label{thm:supergradual}
Suppose the conditions in Theorem \ref{thm:gradual} hold for an RSVP.
For any point in $\reals_+^3$, if there exists an optimal path with a finite number of pieces
then there exists a gradual optimal path.
\end{theorem}
\begin{proof}Theorem \ref{thm:gradual} establishes the result for points on
the boundary of the octant. By convexity, the last segment of an optimal
path to an interior point must have an endpoint on the boundary. The result
then follows directly from Theorem \ref{thm:gradual}.
\end{proof}

\section{Results for Exotic Paths}  \label{sec:exotic}

This entire section is devoted to arguing that certain ``exotic spirals'' cannot be optimal. Depicted in the left-hand side of Figure \ref{exotic_pic} is
what we call a classic spiral, a path type which has appeared in other contexts in the
literature on fluid models. In Section \ref{sec:optspiral} we show that such a path can
indeed be the optimal solution to the variational problem we consider in the paper.
For now, however, we wish to show that other types of paths, exotic spirals, cannot be
better than a classic spiral. One important type of exotic spiral appears in
the right-hand side of Figure \ref{exotic_pic}. Eliminating this type of path from
consideration is the focus of much of the next results.

\begin{figure}
\begin{centering}
 \includegraphics[height=7cm]{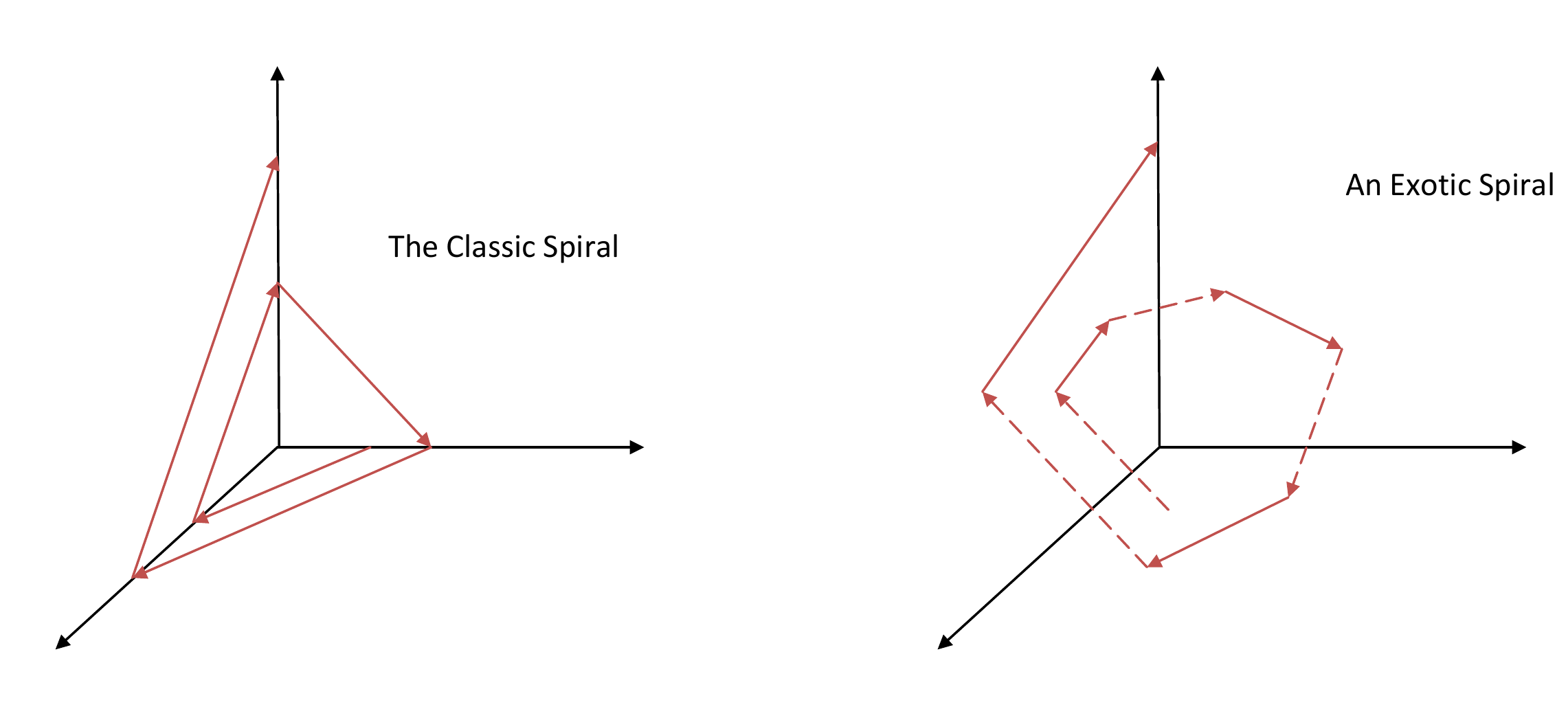}
\caption{Spiral paths} \label{exotic_pic}
\end{centering}
\end{figure}

\begin{lemma} \label{parallel}
For any two-dimensional face, define the \emph{bisecting ray} to be the ray which forms an
angle of $\pi/4$ radians with the adjacent axes and whose endpoint is the origin.
Consider an optimal path which the following characteristics: it contains a line segment that intersects the bisecting ray in $F_1$
and it contains another line segment that intersects the bisecting ray in $F_2$. Then
there exists an optimal path with the following characteristics:
\begin{itemize}
\item[(a)] The path has two segments (as defined above) which form the same angle with the bisecting rays (i.e., if the segments are rotated to lie
in the same face, then they must be parallel).
\item[(b)] The path contains another segment in $F_3$ which intersects the bisecting
ray at the same angle.
\item[(c)] The path contains an infinite number of segments.
\end{itemize}
\end{lemma}
\begin{proof}
Let $\overline{v^{0}v^{1}} \in F_{1}$ and $\overline{v^{2}v^{3}} \in F_{2}$ be the segments which
intersect the respective bisecting rays and suppose the points are traversed by the optimal path in the
order $v^0$, $v^1$, $v^2$, and $v^3$.

We prove part (a) by contradiction, assuming the segments do not form the same angles with
the bisecting rays.
Consider then the portion of the path from the origin to $v^1$.
This portion can be rotated and scaled to create an optimal path which passes through the
point, call it $w$, where $\overline{v^{2}v^{3}}$ intersects the bisecting ray in $F_2$.
Thus, we can create a path from the origin through $w$ to $v_3$ which is optimal yet
has a ``kink'' at $w$. However, this path cannot be optimal due to reflected convexity
(Lemma \ref{refconvex}). This establishes (a).

It is clear by rotation, scaling, and merging that one can form an optimal path to $v_3$
which intersects the bisecting ray in $F_3$. Repeating the process results in the formation
of an optimal path with an infinite number of such segments. This establishes (b) and (c).
\end{proof}

We present one more lemma before giving the main result which eliminates
exotic spirals from consideration. 

\begin{lemma}
\label{thm:newspiral}
Consider an RSVP with $\theta_0 < 0$. Define $v = (v_1, 0, v_3)$ and $u = (u_1, u_2, 0)$, where $v_1, v_3, u_1, u_2 > 0$. Set $v' = (kv_3, kv_1, 0)$, and $u' = (0, ku_1, ku_2).$ Then, for all $k \in (0,1)$  the three-piece path passing through $u'$, $v'$,
$u$, and $v$, with cost $\tilde{\mathcal{I}}_{0}(u',v') +
\tilde{\mathcal{I}}_{3}(v',u) + \tilde{\mathcal{I}}_{0}(u,v)$ is suboptimal. 
\end{lemma}
\begin{proof}

Define $u_x = (u_1 - x, u_2 - x, 0)$ and $v'_x = (kv_3 - x, kv_1 - x, 0)$ for some $x > 0$. Since $v_1, v_3, u_1, u_2 > 0$, when $x$ is small enough, $u_x, v'_x \in int(F_{3})$ for all $x$ in a non-negative neighborhood of 0. For all such $x$, $\tilde{\mathcal{I}}_{3}(v',u) = \tilde{\mathcal{I}}_{3}(v'_x,u_x).$ In the same non-negative neighborhood, define $$f(x) = \tilde{\mathcal{I}}_{0}(u', v'_x) + \tilde{\mathcal{I}}_{0}(u_x,v).$$ To establish the result, it is enough to show that $f'(0) < 0.$
\\
We have that $$f(x) =  \|\theta\| (\| v'_x -u'\|+ \|v-u_x \| ) - \langle \theta, v'_x-u'+v-u_x \rangle.$$
So straightforward calculations yield
\begin{equation*}
\begin{aligned}
&f'(x) = \sigma ^{-2} \|\theta\| \left[\frac{\gamma_0(2x+v_1-u_1-u_2)+\gamma_1(2x+2v_3+v_1-u_1-u_2)}{\|v-u_x\|} \right. \\
 + &  \left.   \frac{\gamma_0(2x+ku_1-kv_1-kv_3) + \gamma_1(2x+2ku_2+ku_1-kv_3-kv_1)}  {\|v_x - u'\|}\right].
\end{aligned}
\end{equation*}
Recalling that $\gamma_0 > \gamma_1$ by Lemma \ref{lemma:matrix}, we obtain
$$f'(0) = -\frac{\|\theta\|}{\|v-u_x\|}\sigma ^{-2}(\gamma_0 -\gamma_1)(u_2+v_3) < 0, $$
establishing the desired result.
\end{proof}

\begin{theorem}[Elmination of Exotic Spirals]  \label{thm:ex}
Consider an RSVP with $\Gamma = I$, $\theta_0 < 0$, $r_1,r_2 \ge 0$ and suppose Condition 1 holds.
For any optimal path to a point on the axis with a countably infinite number of segments, there exists
another path, with lower or equal cost, which is of the form of the classic spiral (i.e., of the form $(A,A,A, \ldots, O)$).
\end{theorem}
\begin{proof}We begin with a general principle that holds for paths with an infinite number
of segments. Consider an optimal path characterization which begins with an $A$ and contains another
$A$  at position $n$, elsewhere in the sequence. Then there exists an optimal
path whose entire characterization
must be identical to the (original) characterization starting at position $n$. This principle follows
directly by scaling, rotation, and merging and it can be thought of as enforcing a ``self-similarity'' property of optimal paths.
As an example, consider an optimal path of the form $(A, AS, A, A, A, A, \ldots, O)$ where $AS$ is an arbitrary subsequence. The principle implies that such a path can be replaced
by a classic spiral of the form $(A,A,A, \ldots, O)$.

Now, consider the terminal point of an optimal path, which by assumption lies on an axis and which
by our convention is represented by the first $A$ in the sequence characterizing this path.
If the next endpoint lies on an axis, then the path is a classic spiral (or can be replaced by one),
based on the principle above. So suppose this is not the case. The second endpoint cannot
be on a different face $(DF)$ by the Bad Faces Theorem (Theorem \ref{theorem:shortest}). Thus the only remaining
possibility is that the second endpoint is characterized as $SF$, that is, it lies on the interior of
one of the two adjoining faces.

Next, in any place in the sequence only a $DF$ can follow $SF$ by convexity and the Bad Faces Theorem.
After a $DF$, either an $SF$ or an $A$ may follow ($DF$ cannot follow, again by convexity). Finally,
between any two appearances of an $A$ in the sequence, the $SF$ and $DF$ sequences can be assumed
to be the same, again invoking the self-similarity principle above. Putting all of these observations together, we conclude
that apart from the classic spiral case, there are only two other general categories of paths with an infinite number
of segments:
\begin{itemize}
\item[(i)] $(A,SDF,O)$, where $(SDF)$ is an infinite subsequence of $(SF,DF)$.
\item[(ii)] $(A,SDF_{i}, A,SDF_{i}, A, SDF_{i}, \ldots, O)$, $i=1,2,3,\ldots.$
\end{itemize}
We now proceed to eliminate these two types of paths.

\textbf{Part (i).} We consider first a path of type $(A,SDF,O)$. Without loss of generality, assume that the
terminal point (the first $A$ in the sequence) is $e_3$ and the next pivot point is
$v^1 \in int(F_2)$. Now the farthest face from $v^1$ can be either $F_1$ or $F_3$
(since the point is on $F_2$, this cannot be the farthest face). Suppose $v^1$ is strictly
closer to $F_1$. Then by the Bad Faces Theorem, the next endpoint must be in $F_1$.
However, such a path can be eliminated from consideration by Lemma \ref{axis_eliminate}.
Therefore we assume that $v^1$ is closer to $F_3$ than $F_1$ and the next endpoint
in the path is in $F_3$ again by the Bad Faces Theorem. (If $v^1$ is equidistant to $F_1$
and $F_3$, then the segments to $F_1$ and $F_3$ have the same costs and we choose
the segment going to $F_3$). The next point, $v^3$ is also in the interior of $F_3$ due to our assumption
on the path type. Using arguments from the proof of the Bad Faces Theorem, it
can be shown that $v^2$ must be closer to $F_2$ than $F_1$. 
Next, if $v^3$ is closer to $F_2$ than $F_1$ then the resulting path is
of ``switchback'' form. Such a path is suboptimal by Lemma  
\ref{lem:switchback}. Thus, $v^3$ must be closer to $F_1$ than $F_2$. 
Furthermore, by Lemma \ref{parallel}, $\overline{e_{3}v^{1}}$ is rotationally parallel to
$\overline{v^{2}v^{3}}$. Hence, after $e_3$ the faces containing the endpoints
are in this order: $F_2, F_3, F_3, F_1, F_1, F_2, F_2, \ldots.$
By the usual arguments using rotation and scaling, all the $DF$ segments are rotationally
parallel. So, for
this path type, the path is an ``exotic spiral'' as depicted in the right half of Figure \ref{exotic_pic}. However, Lemma \ref{thm:newspiral} implies that such an exotic spiral is suboptimal. 

\textbf{Part (ii).} We now turn our attention to the other general type of exotic spiral, one with the characterization
 $(A,SDF_{i}, A,SDF_{i}, A, SDF_{i}, \ldots, O),$ $i \in \{1,2,3,\ldots \}$ where each $SDF_i$ is
 a sequence with $SF/DF$ segments repeated $i$ times. By symmetry and scaling arguments,
 we can assume that each of these sequences is identical. 
 Suppose that the end point $v^0$ is on axis $F_{1,2}$ and the SF is from $v^1$ on $F_1$.  By Lemma \ref{axis_eliminate}, since $v^{2}$ is on a different face than $v^1$, $v^1$ must be closer to $F_{1,3}$ than to $F_{1,2}$ and $v^2$ must be in $F_3$. The next point, $v^3$, is either on axis $F_{2,3}$ or in $F_{3}$. However, it must be closer to $F_{2,3}$ than to $F_{1,3}$. Then based on Lemma \ref{parallel}, there exists point $u$ on axis $F_{1,3}$ for which an optimal path to $u$ contains $\overline{v^{3}v^{2}}$. By scaling and symmetry, this path can be assumed to be of the $(A,SDF_{i}, A,SDF_{i}, A, SDF_{i}, \ldots, O)$ form posited for the path to $v_0$. 
 Now, by assumption, there are $i$ $SF/DF$ segments between $v^0$ and the next point on the axis
 and of course this path passes through $v_2$. Considering the optimal path to $u$, since
 it also passes through $v_2$ en route to $u$, the portion of the this path to $v_2$ can
 be replaced by the optimal path to $v_0$, up to $v_2$. This patching process forms
 another optimal path to $u$. Since there are $i$ $SF/DF$ segments between $v^0$ and the next
 axis point, there are then $i-1$ $SF/DF$ segments between $u$ and this same point.
 Therefore we have constructed an optimal path to a point on the axis ($u$) which is of the
 form $(A,SDF_{i-1}, A,SDF_{i-1}, A, SDF_{i-1}, \ldots, O)$. Hence, there must exist
 an optimal path of the same form to $v_0$. Repeating this patching process results in the
 construction of an optimal path to $v_0$ (and thus any point on any axis) of the
 form  $(A, A, A, \ldots,0)$. 
 
 It remains to be argued that there must be a finite number of segments between
in an optimal path between any two axis points. We only give an outline here. 
Consider a path with an infinite number of segments which converge to a point
$v_0$ on $F_{1,2}$. There must be an infinite
number endpoints of such segments in an $\epsilon$-ball around $v_0$. Furthermore, by
Lemma \ref{lem:switchback} there must exist an infinite subsequence of 
endpoints for which the other terminal point of the segment is in $F_3$. 
The cost of all such segments can be uniformly bounded away from zero 
(using the infimum of the cost from the $\epsilon$-ball to $F_3$, which
is strictly positive). This implies however, that the total cost of any such path
is infinite. Hence, the path cannot be optimal. 
\end{proof}

\section{An Optimal Spiral Path} \label{sec:optspiral}
In Example 2 of Section 6 of El Kharoubi et al.\ \cite{kytb11}, it is shown
that a spiral path has a lower cost than a two-piece gradual
path, for the corresponding RSVP.  Here we give a related example, and using the results of previous sections,
show that a spiral path is indeed optimal.

Let $\theta = (-1,-1,-1)'$, $\Gamma=I$ and $$ R =\left(
                                                                     \begin{array}{ccc}
                                                                       1 & 0 & \frac{3}{2} \\
                                                                       \frac{3}{2} & 1 & 0 \\
                                                                       0 & \frac{3}{2} & 1 \\
                                                                     \end{array}
                                                                   \right). $$
To establish that a spiral path is optimal, we need to undertake three steps.
First, we check that Condition 1 of Section \ref{sec:finite} holds. If so, then we know
that only gradual paths or spiral paths are optimal. Second,
using results from \cite{kytb11} we check the reflectivity characteristics of optimal
paths traversing an axis. Third, to travel to a point, say $e_3$, on an axis, we verify
that it is less costly to traverse one of the other axes and then cross a two-dimensional
face. If this is the case, then one can construct a spiral path to $e_3$ that is cheaper
than the gradual path to $e_3$ (which simply travels along the axis).

The next proposition simplifies the process of checking Condition 1 and
may be useful in producing other examples.
\begin{proposition} For a RSVP with
$r_2 = 0$ and $0< r_1 < 2$
Condition 1 holds.
\end{proposition}

\begin{proof}
Clearly $(r_1,r_2) \in R_f$ under the assumptions given.
Recall that Condition 1 is given by
\begin{equation}
(1+r_{2}^2)(1+r_{1}^2-r_{2}-r_{1}r_{2})^{2} \geq 2(r_{1}r_{2})^{2}(1+r_{1}^{2}+r_{2}^{2}-r_{1}-r_{2}-r_{1}r_{2}).
\end{equation}
For $r_2=0$ the condition reduces to $(1+r_1^2)^2 \ge 0$ which clearly holds for all real $r_1$.
\end{proof}
So, the proposition provides verification of Condition 1 for the example in this section.
Next, we use results from \cite{kytb11} to check reflectivity of the axes. In
particular we use equations (24) and (25), Remark 2, and Proposition 1 from
that paper.
Let $R_{1}=(1,3/2,0)'$, $R_{2}=(0,1,3/2)'$, and  $R_{1,2}= (R_{1}, R_{2})$.
Define $A_{1,2}=I-R_{1,2}B_{1,2}$ and $B_{1,2}=(R_{1,2}'R_{1,2})^{-1}R_{1,2}'$.
Some algebra shows that $$\frac{\|A_{1,2}\theta \|}{\|A_{1,2}e_{3}\|}B_{1,2}e_{3}-B_{1,2}\theta
\approx (0.0526, 1.5526)  > 0.$$ So optimal one-piece reflected paths confined to an axis use both corresponding reflection vectors.

Finally, we check the spiral condition for the point $e_3$. In particular either
$$\tilde{\mathcal{I}}_{1,2}(e_{3}) \geq \tilde{\mathcal{I}}^2_{\{2,3\},2}(e_{3})
\quad \mbox{or} \quad  \tilde{\mathcal{I}}_{1,2}(e_{3}) \geq \tilde{\mathcal{I}}^2_{\{1,3\},1}(e_{3})$$
must hold. We verify the first inequality.
For the parameters of our example we have $$\tilde{\mathcal{I}}_{1,2}(e_{3}) = \|A_{1,2}\theta \| \|A_{1,2}e_{3}\|- \langle A_{1,2}\theta, A_{1,2}e_{3} \rangle \approx 0.4211.$$
Next, let $u=(0.5,0,0)$. Then $$\tilde{\mathcal{I}}^2_{\{2,3\},2}(e_{3}) \leq \tilde{\mathcal{I}}_{2,3}(u)+\tilde{\mathcal{I}}_2(e_{3}-u) \approx 0.3317 <\tilde{\mathcal{I}}_{1,2}(e_{3}).$$
Therefore, for the given RSVP, the optimal path to any point on the boundary of the octant is a classic
spiral optimal path.

We can more precisely characterize this optimal spiral path. In particular, the last piece connects
the points $ke_1$ and $e_3$, where $k$, $0 < k < 1$, is the shrink factor. The optimal
value of $k$ can be calculated by defining the corresponding spiral cost as a function
of $k$:
$$f(k) = \frac{\tilde{\mathcal{I}}_{2}(e_3 - ke_1)}{1-k}. $$
Applying the data for this problem and setting $f'(k)=0$ results in the quadratic root-finding problem
 $1228123k^2 - 3690960k + 1626300 = 0$. The appropriate root is
 $k^* \approx 0.5363.$
 From this we can calculate the cost of the optimal spiral as $$\frac{\tilde{\mathcal{I}}_{2}(e_3 - k^*e_1)}{1-k^*} \approx 0.2384.$$

\section{Conclusions}
As mentioned in the introduction, this paper only provides a piece of the puzzle of the
variational problem related to large deviations for SRBM in the orthant. We have only
addressed problems with symmetric data, and even then some of the results have further
restrictions on the parameters. Although we do not provide an analytical
proof, in Liang \cite{liang12} convincing numerical evidence indicates that 
Condition 1 holds whenever the SRBM is stable and $r_1, r_2 \ge 0$. Other results 
require that the covariance matrix $\Gamma$ is the identity. The covariance
condition seems more difficult to remove, since this matrix could affect
the types of paths
that are optimal for a given SRBM. As noted earlier, it is already know that our results cannot be generalized to arbitrary (stable)
SRBM data. The example
in \cite{dupram02} which was discussed in the introduction is particularly troubling
because the reflection matrix in that example is partially rotationally symmetric
(two of the three reflection vectors exhibit rotationally symmetry).

In addition, even once one has a handle on the types of paths which are optimal,
computation and comparison of these path costs appear still requires considerable
effort. This indicates that fully solving large deviations problems in high dimensions
is likely to remain a challenge and future work in dimensions four and higher will
require examination of specialized cases and increasing mathematical creativity.

\textbf{Acknowledgments.} We would like to thank Jim Dai, Martin Day, Kasie Farlow, Michael Harrison,
and Kavita Ramanan for enlightening discussions on this problem. 
Our appreciation also goes out to the anonymous review team which 
provided excellent suggestions for improvement. 
Finally, we are indebted to Ahmed El Kharroubi
for sharing a preprint of his paper with us. 

\bibliographystyle{abbrv}

\section{Appendix}

\begin{proof}[Proof of Lemma \ref{lemma:matrix}]
 It can be checked that if $\Gamma^{-1}$ has the form of the lemma, then for a rotationally symmetric matrix $\Gamma$ the nine equations
in $\Gamma \Gamma^{-1} = I$ are consistent. Since matrix inverses are unique, it follows immediately that $\Gamma^{-1}$ can be written as stated in the lemma. Next, in order for $\Gamma \Gamma^{-1} = I$ to hold
we must have
\begin{eqnarray*}
\gamma_{0}+2\rho\gamma_{1} & = & 1 \qquad \mbox{and} \\
\gamma_{1}+\rho(\gamma_{0}+\gamma_{1}) & = & 0.
\end{eqnarray*}
Solving these equations yields
\begin{equation} \label{eqn:rho}
\rho=-\frac{\gamma_{1}}{\gamma_{0}+\gamma_{1}}=\frac{1-\gamma_{0}}{2\gamma_{1}},
\end{equation}
which implies
\begin{equation}  \label{eqn:gam}
2\gamma_{1}^{2}=(\gamma_{0}-1)(\gamma_{0}+\gamma_{1}).
\end{equation}
Note that $\gamma_0 = 0$ is not possible. To prove $\gamma_0 > \gamma_1$ we examine
four cases.

\begin{enumerate}
\item If $\gamma_{0}>0$ and  $\gamma_{1} \leq 0$ then the result follows immediately.
\item Suppose $\gamma_{0}>0, \gamma_{1}>0$ and $\gamma_{0} \le  \gamma_{1}$. Then we have
$$(\gamma_{0}-1)(\gamma_{0}+\gamma_{1}) \le (\gamma_{1}-1)(\gamma_{1}+\gamma_{1})<2\gamma_{1}^{2}.$$ This contradicts (\ref{eqn:gam}).
\item Suppose $\gamma_{0}<0, \gamma_{1} \leq 0$ and $\gamma_{0} \le \gamma_{1}$.
Then we have
$$(\gamma_{0}-1)(\gamma_{0}+\gamma_{1}) \ge (\gamma_{1}-1)(\gamma_{1}+\gamma_{1})>2\gamma_{1}^{2}.$$
This again contradicts (\ref{eqn:gam}).
\item Suppose finally that  $\gamma_{0}<0, \gamma_{1}>0$ and  $\gamma_{0} \le \gamma_{1}$.
Since $\rho < 1$ by definition, (\ref{eqn:rho}) implies that $-2\gamma_{1}+1 < \gamma_0 < 0$.
Solving (\ref{eqn:gam}) gives
$$\gamma_{0}=\frac{1-\gamma_{1}-\sqrt{9\gamma_{1}^{2}+2\gamma_{1}+1}}{2}<\frac{1-\gamma_{1}-3\gamma_{1}}{2}
=-2\gamma_{1}+1/2,$$
which contradicts $-2\gamma_{1}+1 < \gamma_0$. In solving (\ref{eqn:gam}) we take the smaller
root, since we must have $\gamma_0 < 0$.
\end{enumerate}
Thus, by contradiction, we have established $\gamma_{0} > \gamma_{1}.$
\end{proof}

\begin{proof}[Proof of Lemma \ref{lem:abc}]
It can be checked that if $R^{-1}$ has the form of the lemma, then for a rotationally symmetric matrix $R$ the nine equations
in $R R^{-1} = I$ are consistent. Since matrix inverses are unique, it follows immediately that $R^{-1}$ can be written as stated in the lemma.
Also, from $R R^{-1} = I$, we have
\begin{eqnarray*}
a+r_{2}c+r_{1}b & = & 1 \\
b+r_{2}a+r_{1}c & = & 0 \\
c+r_{2}b+r_{1}a & = & 0.
\end{eqnarray*}
Summing these equations gives the claimed equality.
\end{proof}

\begin{proof}[Proof of Theorem \ref{pands}]
Lemma \ref{comps} states that
$R$ being completely-$\cal{S}$ is equivalent to
\begin{equation} \label{s}
r_{1}+r_{2}+1 > 0.
\end{equation}
By definition, the conditions
\begin{eqnarray}
1-r_{1}r_{2} & > & 0 \label{p1} \\
1+r_{1}^{3}+r_{2}^{3}-3r_{1}r_{2} & > & 0 \label{p2}
\end{eqnarray}
are necessary
and sufficient for $R$ to be a $\cal{P}$-matrix.
We prove that (\ref{s}) is equivalent to (\ref{p1}) and
(\ref{p2}) by partitioning the possible values of $r_1$ and $r_2$.

\begin{enumerate}
\item Suppose that either $r_{1}=0$ or $r_{2}=0$.
We prove the case in which $r_2=0$, the other case is analogous.
When $r_2=0$ (\ref{s}) reduces to $r_1+1>0$, (\ref{p1}) is trivially
satisfied, and (\ref{p2}) reduces to $r_1^3+1>0$. The first inequality
is equivalent to the last, establishing the result for this case.

\item Suppose $r_{1},r_{2} > 0$. It is easy to see that (\ref{s}) always
holds in this case. Furthermore, $r_1+r_2 <2$ implies
$$\left(\frac{r_{1}+r_{2}}{2}\right)^{2} < 1.$$
The arithmetic-geometric mean (AGM) inequality gives
\begin{equation}
r_{1}r_{2} \leq \left(\frac{r_{1}+r_{2}}{2}\right)^{2},
\end{equation}
and thus (\ref{p1}) always holds.
Invoking the AGM inequality again yields
$$ r_{1}r_{2} = \sqrt[3]{r_{1}^{3}r_{2}^{3}} \le \frac{1+r_{1}^{3}+r_{2}^{3}}{3}.$$
Recall that equality in the AGM inequality holds iff the three terms are equal.
Equality of the terms in this case implies $r_1=r_2=1$ which is not possible due
to $r_1+r_2 <2$. Therefore (\ref{p2}) automatically holds.

\item Suppose $r_{1},r_{2} < 0$.
Note that
\begin{equation}  \label{pos}
1+r_{1}^{3}+r_{2}^{3}-3r_{1}r_{2} = (r_{1}+r_{2}+1)(1+r_{1}^{2}+r_{2}^{2}-r_{1}-r_{2}-r_{1}r_{2})
\end{equation} and
$$1+r_{1}^{2}+r_{2}^{2}-r_{1}-r_{2}-r_{1}r_{2} > 0,$$
when $r_{1},r_{2} < 0$. Therefore, in this case,
(\ref{s}) and (\ref{p2}) are equivalent and we need only show that (\ref{s}) implies
(\ref{p1}). Note then that $r_{1}+r_{2}+1 > 0$, implies that $r_{1} > -1$ and  $r_{2} > -1$.
Given that both $r_1$ and $r_2$ are also negative this yields $r_{1}r_{2} <1$.

\item Suppose $r_{1} > 0$ and $r_{2} < 0$ (the case $r_1<0$, $r_2>0$ is analogous).
It is obvious that (\ref{p1}) always holds in this case and so we need only
show that (\ref{s}) and (\ref{p2}) are equivalent.
Consider again the last term in (\ref{pos}):
$$1+r_{1}^{2}+r_{2}^{2}-r_{1}-r_{2}-r_{1}r_{2} = (r_{1}-r_{2})^{2}-(r_{1}-1)(1-r_{2}).$$
Note that when $r_1 < 1$ this term is positive as can be seen from the right-hand side
above. When $r_1 \ge 1$ we have $r_{1}-r_{2} > r_{1}-1 \ge 0$ and $r_{1}-r_{2} \ge 1-r_{2} > 0$
and again the right-sand side above is clearly positive. This fact implies that
(\ref{s}) and (\ref{p2}) are equivalent, as argued in Case 3.

\end{enumerate}
\end{proof}

\begin{lemma}[Reflected Convexity] \label{refconvex}
Consider a section of a feasible triple $(x,y,z)$ in which the path $z$ consists of
segments $\overline{v^1v^2}$ and $\overline{v^2v^3}$, with
$v^1,v^2, v^3 \in F_j$ where $j \in \{1,2,3\}$. Suppose that
$\overline{v^1v^2}$ is a reflected segment and $\overline{v^2v^3}$
is direct.  Then
there exists a linear reflected path from $v_1$ to $v_3$ whose cost
is no greater than the original path.
\end{lemma}
\begin{proof}
Without loss of generality, assume that $j=1$. Let $(x^1(t), y^1(t), z^1(t))$ be the  triple
corresponding to the segment $\overline{v^1v^2}$ with $T=T^1$. Similarly, let
$(x^2(t), y^2(t), z^2(t))$ be the triple corresponding to $\overline{v^2v^3}$ with $T=T^2$.
Note that $(x^1(t), y^1(t), z^1(t)) = (\dot{x}^1, \dot{y}^1, \dot{z}^1)t$ and $(x^2(t), y^2(t), z^2(t))=(\dot{x}^2, \dot{y}^2, \dot{z}^2)t.$ Further denote $\dot{x}^1 = (x^1_{1}, x^1_{2},x^1_{3})'$. Similar notation is used for the other variables. Notice that $\dot{y}^1 = (y^1_{1}, 0, 0)'$ and $\dot{y}^2 = 0.$

By our assumptions on the segments, we have  $\dot{x}^1 +R \dot{y}^1 = \dot{z}^1$, $\dot{x}^2 = \dot{z}^2$,
$\dot{z}^1_{1} = 0$, $\dot{x}^1_{1} < 0$ and $\dot{x}^2_{1} = 0.$ 
By translation, we set $z^{1}(T^{1}) = v^{2}-v^{1}$ and $z^{2}(T^{2}) = x^{2}(T^{2}) = v^{3}-v^{2}.$  Also, define points $u^{2} = v^{1}+x^{1}(T^{1})$ and $u^{3} = u^{2}+x^{2}(T^{2})$. Notice that these two points are not in the interior of the octant. Based on convexity, $$\tilde{\mathcal{I}}_{0}(u^{2},u^3)+\tilde{\mathcal{I}}_{0}(v^{1},u^2) \geq \tilde{\mathcal{I}}_{0}(v^{1},u^3).$$
Let $x^3(t)$ be optimal to $\tilde{\mathcal{I}}_{0}(v^{1},u^3)$ with corresponding $T = T^3$, where $x^{3}(t) = \dot{x}^{3}t.$
It is clear that $\dot{x}^{3}T^{3} = u^{3}-v^{1} = \dot{x}^{1}T^{1}+\dot{x}^{2}T^{2}.$ Define $\dot{y}^{3} = \frac{\dot{y}^{1}T^{1}}{T^{3}}$ and $y^{3}(t) = \dot{y}^{3}t.$ Also define $z^{3}(t) = z^{3}(t)+Ry^{3}(t).$ So $z^{3}(T^{3}) = v^{3}-v^{1}.$ Thus $(x^{3}(t), y^{3}(t), z^{3}(t))$ is a feasible triple for $\tilde{\mathcal{I}}_{1}(v^{1},v^3).$
Therefore, $$\tilde{\mathcal{I}}_{1}(v^{1},v^3) \leq \frac{1}{2} \int_0^{T^3} || \dot x^{3}(t) - \theta||^2\, dt = \tilde{\mathcal{I}}_{0}(v^{1},u^3) \leq \tilde{\mathcal{I}}_{0}(u^{2},u^3)+\tilde{\mathcal{I}}_{0}(v^{1},u^2) = \tilde{\mathcal{I}}_{1}(v^{1},v^2) + \tilde{\mathcal{I}}_{0}(v^{2},v^3),$$
which establishes the result.

\end{proof}

\begin{proof}[Proof of Lemma \ref{lem:case3}]

\textbf{Part (a).} 
First, without loss of generality we set $v=(v_1,0,1)$.
We prove that if Condition 1 holds, then
 $\tilde{\mathcal{I}}_2(v) > \tilde{\mathcal{I}}_{2}(e_{3}).$
 Of course, this immediately implies that 
 $\tilde{\mathcal{I}}_2(v) > \tilde{\mathcal{I}}_{2}(ke_{3}),$
 for all $k \in [0,1]$.
 
For all non-negative $v_1$, define the function
$$G(v_1):=\tilde{\mathcal{I}}_2(v)=\|Av\| \|A\theta\|-\langle Av, A\theta \rangle,$$
where $A=I-R_{2}B$, $B=(R_{2}'R_{2})^{-1}R_{2}'$, and $R_{2}=(r_{2},1,r_{1})'$. It can
be checked that $G(\cdot)$ is strictly convex on $(0,1)$.
Therefore, to prove $\tilde{\mathcal{I}}_2(v) > \tilde{\mathcal{I}}_2(e_{3})$
for $v_1 >0$,
 it
is enough to show that $\frac{\partial_+  G(v_1)}{\partial v_1}|_{v_1=0} \ge 0$. Some algebra yields
$$\left. \frac{\partial_+ G(v_1)}{\partial v_1}\right|_{v_1=0}=\frac{1}{2}\frac{\|A\theta\|}{\|Ae_{3}\|}(A_{31}+A_{13})-(A\theta)_{1}.$$
Note that $A_{31}+A_{13} \leq 0$ and $(A\theta)_{1} \leq 0$ in $R_{f}$.
So, to prove the non-negativity of the derivative, it is sufficient to show that
\begin{equation} \label{derv}
(A\theta)_{1}^{2} \geq  \frac{1}{4}\left[\frac{\|A\theta\|}{\|Ae_{3}\|}(A_{31}+A_{13})\right]^{2}.
\end{equation}
Next, we have
 $$\left(\frac{\|A\theta\|}{\|Ae_{3}\|}\right)^2 = \frac{2(1+r_{1}^{2}+r_{2}^{2}-r_{1}-r_{2}-r_{1}r_{2})}{1+r_{2}^2}\theta_{0}^{2},$$ and $$(A\theta)_{1}=\frac{1+r_{1}^2-r_{2}-r_{1}r_{2}}{1+r_{1}^2+r_{2}^2}\theta_{0}.$$
Plugging these equalities into (\ref{derv}) yields the condition
\begin{equation} \label{rfinal}
(1+r_{2}^2)(1+r_{1}^2-r_{2}-r_{1}r_{2})^{2} \geq 2(r_{1}r_{2})^{2}(1+r_{1}^{2}+r_{2}^{2}-r_{1}-r_{2}-r_{1}r_{2}).
\end{equation}
In summary, if (\ref{rfinal}) holds, then $\frac{\partial_+ G(v_1)}{\partial v_1}|_{v_1=0} \ge 0$ which in turn
implies $\tilde{\mathcal{I}}_2(v) > \tilde{\mathcal{I}}_2(e_{3})$.

\textbf{Part (b).} The claim is that if the conditions of the lemma statement hold then $\tilde{\mathcal{I}}_{0}(v,v') \geq \tilde{\mathcal{I}}_1(ke_{3},v')$ for some
$k \in [0,1]$.
First consider the case when $ k =1$.
Since $\tilde{\mathcal{I}}_{0}(v,v')$ and $\tilde{\mathcal{I}}_1(e_{3},v')$ are both proportional to $\theta_{0}$ it is enough to verify this case when $\theta_{0} =-1$. We have that $\tilde{\mathcal{I}}_{0}(v,v') = \frac{1}{2}\|\dot{x}^{*}(t)-\theta\|^{2} T^{*}$ where
$$T^{*} = \frac{\|v'-v\|}{\|\theta \|} = \sqrt {\frac{(v_{1})^2+(v_{2})^{2}+(v_{3}-1)^{2}}{3}}$$ and $x^{*}(t) =x^{*}t=t(x_{1}^{*}, x_{2}^{*}, x_{3}^{*})'$ for $t \in [0,T]$.
We construct a feasible reflected path contained in $F_1$ from $e_3$ to $v'$ with a cost
$H_{\tilde{x}}(e_3,v')$
less than or equal to $\tilde{\mathcal{I}}_{0}(v,v')$. This construction then implies
$\tilde{\mathcal{I}}_{0}(v,v') \geq \tilde{\mathcal{I}}_1(e_{3},v').$

Denote a feasible triple from $e_{3}$ to $v'$ by $(\tilde{x}(t),\tilde{y}(t),\tilde{z}(t))$ on $[0,\tilde{T}]$.   Let $\tilde{T}=T^{*}$, $\tilde{z}(t)=\tilde{z}t=t(0,\tilde{z}_{2},\tilde{z}_{3})'$ and $\tilde{y}(t) = t (\tilde{y}_1,0,0)'$ for some $\tilde{y}_1 \geq 0$. It is clear that $\tilde{z}_{2}=x_{2}^{*}$ and $\tilde{z}_{3}=x_{3}^{*}$. The goal now is to determine if there exists a $\tilde{y}_1 \geq 0$ such that $$\tilde{\mathcal{I}}_{0}(v,v')-H_{\tilde{x}}(e_{3},v') \geq 0.$$
Plugging in $\tilde{x}(t)=\tilde{z}(t)-R\tilde{y}(t)$, $x^{*}T^{*} = v'-v$, and writing $\bar{y}=\tilde{y}_1 T^{*}$ we see that the inequality above is equivalent to
\begin{equation}
\begin{aligned} \label{old2}
&(T^{*}-v_{1})^{2}+(T^{*}+v_{2})^{2}+(T^{*}+v_{3}-1)^{2} \\
&-[(T^{*}-\bar{y})^2+(T^{*}+v_{2}-r_{1}\bar{y})^{2}+(T^{*}+v_{3}-1-r_{2}\bar{y})^{2}] \geq 0.
\end{aligned}
\end{equation}

So, for given problem data and points $v$ and $v'$
if there exists a $\bar{y} \geq 0$ such that (\ref{old2}) is satisfied, then the desired feasible path construction can be achieved.
Notice that the left-hand side (LHS) in this equation is a concave, quadratic function of
$\bar{y}$. So to prove the desired inequality, it is necessary
that this function has a non-negative maximum which is achieved at a
non-negative value.

In (\ref{old2}), the maximum value of  the LHS is reached at $$\bar{y}^{*} = \frac{(1+r_{1}+r_{2})T^{*}+r_{1}v_{2}+r_{2}(v_{3}-1)}{1+r_{1}^{2}+r_{2}^{2}}$$ and this maximum is achieved at a non-negative value when
\begin{equation} \label{old4}
[(1+r_{1}+r_{2})T^{*}+r_{1}v_{2}+r_{2}(v_{3}-1)]^{2}-(1+r_{1}^{2}+r_{2}^{2})(2T^{*}v_{1}-v_{1}^{2}) \geq 0.
\end{equation}
When $v_{3} \geq 1$ it is easy to see that $\bar{y}^{*} \geq 0$. Considering now
the LHS of (\ref{old4}) we have that $$ \mathrm{LHS} \geq (1+r_{1}+r_{2})^{2}(T^*)^{2} -(1+r_{1}^{2}+r_{2}^{2})(T^*)^{2}=(T^*)^{2}(2r_{1}+2r_{2}+2r_{1}r_{2}) \geq 0.$$
So, when $v_3 \geq 1$ we have verified that ${\mathcal{I}}_{0}(v,v') \geq \tilde{\mathcal{I}}_1(e_{3},v'),$ i.e., this part of the lemma holds with $k=1$.

Next, for the case $v_3 < 1$, we prove that if the lemma does not hold
for $k=1$ it must hold for some $k \in [0,1)$. 
When $v_3 < 1$, $[(1+r_{1}+r_{2})T^{*}+r_{1}v_{2}+r_{2}(v_{3}-1)]^{2} \geq [(1+r_{1}+r_{2})T^{*}]^{2}$, which was used to establish (\ref{old4}) in
the $v_3 \ge 1$ case, no longer holds.
However, since $$T^{*}= \sqrt {\frac{(v_{1})^2+(v_{2})^{2}+(v_{3}-1)^{2}}{3}} \geq \sqrt {\frac{(v_{3}-1)^{2}}{3}} = \frac{(1-v_{3})}{\sqrt{3}}$$ and $1+r_1+r_2 > 3r_2$, it is clear that $$(1+r_{1}+r_{2})T^{*}+r_{2}(v_{3}-1) > \frac{3r_2}{\sqrt{3}}(1-v_3) > 0. $$ Thus the following related inequality does hold: $$[(1+r_{1}+r_{2})T^{*}+r_{1}v_{2}+r_{2}(v_{3}-1)]^{2} >
[(1+r_{1}+r_{2})T^{*} + r_{2}(v_{3}-1)]^{2}.$$
Using this result, we can now relax (\ref{old4}). The resulting inequality, \begin{equation*}
[(1+r_{1}+r_{2})T^{*}+r_{2}(v_{3}-1)]^{2}-(1+r_{1}^{2}+r_{2}^{2})(2T^{*}v_{1}-v_{1}^{2}) \geq 0,
\end{equation*}
is now used to obtain sufficient conditions on $r_1$ and $r_2$ to guarantee  (\ref{old4}).
Now, the inequality directly above 
is equivalent to
\begin{equation} \label{new1}
\left[(1+r_{1}+r_{2})+r_{2}\frac{(v_{3}-1)}{T^{*}}\right]^{2}+(1+r_{1}^{2}+r_{2}^{2})\left(-2\frac{v_{1}}{T^{*}}+\left(\frac{v_{1}}{T^{*}}\right)^{2}\right)  = (a-r_2 d)^2 +b(c^2 -2c) \geq 0,
\end{equation}
where $a:=1+r_1+r_2$, $b: = 1+r_1^2 +r_2^2$, $c: = \frac{v_{1}}{T^{*}}$, and $d := \frac{1-v_{3}}{T^{*}}.$ When (\ref{new1}) holds then (\ref{old4}) and thus (\ref{old2}) is satisfied. 

Now we turn to the cases for which $ 0 \leq k < 1$, still assuming $v_3 < 1$.
Construct a feasible triple $(\hat{x}(t),\hat{y}(t),\hat{z}(t)) = t(\hat{x},\hat{y},\hat{z})$ and $\hat{T} = T^{*}$ from $v'$ to $ke_3$ with cost
$H_x(ke_{3},v')$. 
Notice that $x_{1}^{*} = -\frac{v_1}{T^{*}}$ and  $x_{3}^{*} =\frac{v_3 - 1}{T^{*}}.$
Let $\hat{x}_1 = x_{1}^{*}$ and $\hat{x}_3= x_{3}^{*}.$ As $\hat{z}(\hat{T}) = T^{*}\hat{z} = (0, v_{2}, v_{3}-k)'$ it is clear that $\hat{y}_1 = \frac{v_1}{T^{*}} > 0$. Also since $\hat{z} = \hat{x} +R \hat{y}$ we have $$\hat{z}_3 = \frac{v_3 - 1}{T^{*}} +r_{2}\frac{v_1}{T^{*}} = \frac{v_3 - k}{T^{*}}.$$ So $ k = 1-r_2 v_1$. Recalling the stability condition $r_1 + r_2 < 2$ and the assumption of the lemma statement that $r_2 < r_1$ we have $r_2 <1$. Since $v_1 < 1$ also it is clear that $ 0 < k < 1$.
\\
From $\hat{z}_2 = \frac{v_2}{T^{*}}  = \hat{x}_2 +r_{1}\hat{y}_1$ we have $\hat{x}_2 =\frac{v_2-r_1 y_1 }{T^{*}}.$ Since $${\mathcal{I}}_{0}(v,v') = \frac{1}{2}T^{*}[(x_{1}^{*} +1)^2 + (x_{2}^{*} +1)^2 + (x_{3}^{*}+1)^2]$$ and
$$\tilde{\mathcal{I}}_1(ke_{3},v') \leq \frac{1}{2}\hat{T}[(\hat{x}_1 +1)^2 + (\hat{x}_2+1)^2 + (\hat{x}_3 +1)^2]
 = \frac{1}{2}T^{*}[(x_{1}^{*} +1)^2 + (\hat{x}_2 +1)^2 + (x_{3}^{*} +1)^2],$$ it is enough to show 
$$ {\mathcal{I}}_{0}(v,v') - H_x(ke_{3},v') = \frac{1}{2}T^{*}[(x_{2}^{*} +1)^2 -(\hat{x}_2+1)^2] \geq 0,$$ which reduces to
\begin{equation} \label{new2}
(x_{2}^{*} +1)^2 \geq (\hat{x}_2+1)^2.
\end{equation}
 Since $x_{2}^{*} = \frac{v_2}{T^{*}}$ and $\hat{x}_2 =\frac{v_2-r_1 y_1 }{T^{*}}$, (\ref{new2}) is equivalent to
\begin{equation} \label{new3}
 0 \leq r_1 \frac{v_1}{T^{*}}  = r_1 c \leq 2\left(\frac{v_2}{T^{*}} +1\right).
\end{equation}
When this condition is satisfied the inequality ${\mathcal{I}}_{0}(v,v') \geq  
\tilde{\mathcal{I}}_1(ke_{3},v')$ holds for $ k = 1-r_2 v_1$.
If (\ref{new2}) does not hold, then  $$2 \leq 2\left(\frac{v_2}{T^{*}} +1\right) \leq r_1 c \leq 2c$$ which means that $c \geq 1$.
On the other hand, since $$\frac{v_1}{T^{*}} = \frac{v_1}{\sqrt {\frac{(v_{1})^2+(v_{2})^{2}+(v_{3}-1)^{2}}{3}}} \leq \frac{v_1}{\sqrt {\frac{(v_{1})^2}{3}}} = \sqrt{3},$$ 
(\ref{new3}) and thus (\ref{new2})  can be violated only if $r_1 \geq \frac{2}{\sqrt{3}}$. 
\\
Summarizing the arguments so far, we have that if (\ref{new1}) holds then
the lemma is true with $k=1$ and if (\ref{new2}) holds, then the lemma is
true for some $k \in [0,1)$. As a final step, we prove by contradiction that (\ref{new1}) and 
(\ref{new2}) cannot both be false. So, assume that both conditions are violated. 
 If (\ref{new2}) does not hold then $c \geq 1$ and $r_1 \geq \frac{2}{\sqrt{3}}$. As $$ d = \frac{1-v_{3}}{T^{*}} \leq \frac{1-v_{3}}{\sqrt {\frac{(1-v_{3})^2}{3}}} = \sqrt{3},$$ we have $ a-r_2 d \geq 1+r_1+r_2 - \sqrt{3}r_2 > 0.$ Hence $(a-r_2 d)^2$ is decreasing in $d$. Also $b(c^2 -2c)$ is increasing in $c$ when $c > 1.$
When $c = 1$, $d$ reaches its minimum of $\sqrt{2}$. So 
\begin{equation}
\begin{aligned}
&(a-r_2 d)^2 +b(c^2 -2c) > (a-\sqrt{2}r_2)^2-b
= [1+r_1-(\sqrt{2}-1)r_2]^2-(1+r_1^2+r_2^2) \\
> &[1+r_1-(\sqrt{2}-1)(2-r_1)]^2-(1+r_1^2+(2-r_1)^2).
\end{aligned}
\end{equation}
To violate (\ref{new1}) it is necessary to have $$[1+r_1-(\sqrt{2}-1)(2-r_1)]^2-(1+r_1^2+(2-r_1)^2) = (6\sqrt{2}-4)r_1-12(\sqrt{2}-1) < 0$$ which is equivalent to $$r_1 < \frac{12(\sqrt{2}-1)}{6\sqrt{2}-4}.$$ 
However, the right-hand side is smaller than $2/\sqrt{3}$. Since 
$r_1 \geq 2/\sqrt{3}$ is a necessary condition for (\ref{new2}) to be false,
we have reached a contradiction. 
\end{proof}

\end{document}